\documentclass[a4paper]{amsart}
\usepackage{amssymb, amsthm, enumerate}
\usepackage{lscape}
\usepackage{hyperref,caption}
\usepackage{tikz}
\usetikzlibrary{patterns, decorations.markings}
\usepackage{pgfplots}
\pgfplotsset{compat=newest}
\usepgfplotslibrary{fillbetween}
\definecolor{verdxarxadiscriminant}{RGB}{87,146,0}
\definecolor{DarkGreen}{RGB}{32,127,43}
\definecolor{ColorMapViolet}{RGB}{178,178,211}
\newcounter{SubfloatCaptionCounter}[figure]
\graphicspath{{./TheFigures/}}
\makeatletter
\newcommand{\subfigdef}[3][0.49]{\stepcounter{SubfloatCaptionCounter}%
   \begin{minipage}[t]{#1\textwidth}%
      \includegraphics[width=\textwidth]{#2.png}\\%
      \hspace*{0.105\textwidth}\begin{minipage}{0.863\textwidth}\tiny%
          \@tempdimc\textwidth\advance\@tempdimc by -0.75em%
          \setbox\z@\hbox{(\Alph{SubfloatCaptionCounter})}\advance\@tempdimc by -\wd\z@%
          \usebox\z@\hspace*{0.75em}\parbox{\@tempdimc}{#3}
      \end{minipage}
   \end{minipage}
}
\makeatother
\newtheorem{theorem}{Theorem} 
\newtheorem{proposition}[theorem]{Proposition}
\newtheorem{lemma}[theorem]{Lemma}

\theoremstyle{definition}
\newtheorem{remark}[theorem]{Remark}
\providecommand{\abs}[1]{\ensuremath{\left\lvert#1\right\rvert}}
\makeatletter
\newenvironment{labeledlist}[1]{\begin{list}{}{\def\makelabel##1{\hspace*{1em}\bfseries##1:\hfill}%
         \setlength\labelsep{1em}\rightmargin\z@\itemindent\z@\leftmargin\labelsep%
         \setbox\z@\hbox{\makelabel{#1}}\labelwidth\wd\z@\advance\leftmargin by \labelwidth%
         \itemsep=2pt\parsep=0pt\topsep=3pt plus 1pt minus 1 pt}}{\end{list}}
\makeatother
\newcommand{\N}{\ensuremath{\mathbb{N}}}
\newcommand{\R}{\ensuremath{\mathbb{R}}}
\newcommand{\RD}{\mathcal{E}}
\newcommand{\SD}{\mathcal{S}}
\newcommand{\Q}{\mathsf{Q}}
\newcommand{\U}{\ensuremath{\mathbb{U}}}
\title[Chaos in a time-discrete food-chain model]{Dynamics in a
time-discrete food-chain model with strong pressure on preys}
\author{Ll. Alsed\`a$^{1,4,2}$}
\author{J. T. L\'azaro$^{3,2}$}
\author{R. Sol\'e$^{5,6,7}$}
\author{B. Vidiella$^{5,6}$}
\author{J. Sardany\'es$^{4,2}$}

\address{$\phantom{a}^{1}$
  Departament de Matem\`atiques,
  Edifici Cc,
  Facultat de Ci\`encies,
  Universitat Aut\`onoma de Barcelona
  08193 Bellaterra (Barcelona),
  Spain
}
\address{$\phantom{a}^{2}$
  Barcelona Graduate School of Mathematics (BGSMath),
  Campus de Bellaterra,
  Edifici Cc,
  Facultat de Ci\`encies,
  Universitat Aut\`onoma de Barcelona
  08193 Bellaterra (Barcelona),
  Spain
}
\address{$\phantom{a}^{3}$
  Departament de Matem\`atiques,
  Universitat Polit\`ecnica de Catalunya,
  Av. Diagonal, 647,
  08028 Barcelona,
  Spain
}
\address{$\phantom{a}^{4}$
  Centre de Recerca Matem\`atica,
  Campus de Bellaterra,
  Edifici Cc,
  Facultat de Ci\`encies,
  Universitat Aut\`onoma de Barcelona
  08193 Bellaterra (Barcelona),
  Spain
}
\address{$\phantom{a}^{5}$
  ICREA-Complex Systems  Lab,
  Universitat Pompeu Fabra,
  Dr Aiguader 88,
  08003 Barcelona,
  Spain
}
\address{$\phantom{a}^{6}$
  Institut de Biologia Evolutiva,
  CSIC-UPF,
  Pg Maritim de la Barceloneta 37,
  08003 Barcelona,
  Spain
}
\address{$\phantom{a}^{7}$
  Santa Fe Institute,
  1399 Hyde Park Road,
  Santa Fe NM 87501,
  USA
}
\begin{document}
\begin{abstract}
Ecological systems are complex dynamical systems. Modelling efforts
on ecosystems' dynamical stability have revealed that population
dynamics, being highly nonlinear, can be governed by complex
fluctuations. Indeed, experimental and field research has provided
mounting evidence of chaos in species' abundances, especially for
discrete-time systems. Discrete-time dynamics, mainly arising in
boreal and temperate ecosystems for species with non-overlapping
generations, have been largely studied to understand the dynamical
outcomes due to changes in relevant ecological parameters.
The local and global dynamical behaviour of many of these models
is difficult to investigate analytically in the parameter space and,
typically, numerical approaches are employed when the dimension of the
phase space is large.
In this article we provide topological and dynamical results for
a map modelling a discrete-time, three-species food chain with two
predator species interacting on the same prey population.
The domain where dynamics live is characterized, as well as the
so-called escaping regions, for which the species go rapidly to
extinction after surpassing the carrying capacity.
We also provide a full description of the local stability of
equilibria within a volume of the parameter space given by the
prey's growth rate and the predation rates.
We have found that the increase of the pressure of predators
on the prey results in chaos. The entry into chaos is achieved via a
supercritical Neimarck-Sacker bifurcation followed by period-doubling
bifurcations of invariant curves. Interestingly, an increasing
predation directly on preys can shift the extinction of top predators
to their survival, allowing an unstable persistence of the three
species by means of periodic and strange chaotic attractors.
\end{abstract}
\maketitle
\section{Introduction}
Ecological systems display complex dynamical patterns both in space
and time \cite{Sole2006}. Although early work already pointed towards
complex population fluctuations as an expected outcome of the
nonlinear nature of species' interactions
\cite{Elton1924,EltonNicholson1924}, the first evidences of chaos in
species dynamics was not characterized until the late 1980's and
1990's \cite{Constantino1997,Dennis1997}. Since pioneering works on
one-dimensional discrete
models~\cite{May1974,May1976,MayOster1976,Schaffer1986}
and on time-continuous ecological models e.g., with the so-called
spiral chaos \cite{Gilpin1979,Hastings1991}
(already pointed out by R\"ossler in 1976 \cite{Rossler1976}),
the field of ecological chaos experienced a strong debate and a
rapid development
\cite{May1974,May1976,Hastings1991,Schaffer1985,Berryman1989,Allen1993},
with several key papers offering a compelling evidence of chaotic
dynamics in Nature, from vertebrate populations
\cite{Schaffer1984,Schaffer1985,Turchin1993,Turchin1995,Turchin2000,Gamarra2000}
to plankton dynamics \cite{Beninca2008} and insect species
\cite{Constantino1997,Desharnais2001,Dennis1997,Dennis2001}.

Discrete-time models have played a key role in the understanding of
complex ecosystems, especially for those organisms undergoing one
generation per year i.e., univoltine species
\cite{May1974,May1976,Schaffer1986}.
The reason for that is the yearly forcing, which effectively makes
the population emerging one year to be a discrete function of the
population of the previous year \cite{Dennis2001}.
These dynamics apply for different organisms such as insects in
temperate and boreal climates.
For instance, the speckled wood butterfly (\emph{Pararge aegeria}) is
univoltine in its most northern range.
Adult butterflies emerge in late spring, mate, and die shortly after
laying the eggs.
Then, their offspring grow until pupation, entering diapause before winter.
New adults emerge the following year thus resulting in a single
generation of butterflies per year \cite{Aalberg2012}. Hence,
discrete maps can properly represent the structure of species
interactions and some studies have successfully provided experimental
evidence for the proposed dynamics
\cite{Constantino1997,Desharnais2001,Dennis1997,Dennis2001}.

Further theoretical studies incorporating spatial dynamics strongly
expanded the reach of chaotic behaviour as an expected outcome of
discrete population dynamics \cite{Hassell1991,Sole1992}.
Similarly, models incorporating
evolutionary dynamics and mutational exploration of genotypes easily
lead to strange attractors in continuous \cite{Sardanyes2007} and
discrete \cite{Sardanyes2011} time. The so-called \emph{homeochaos}
has been identified in discrete multi-species models with
victim-exploiter dynamics \cite{Ikegami1992,Kaneko1992}.

The dynamical richness of discrete ecological models was early
recognised \cite{May1974,May1976, MayOster1976,McCallum1992} and
special attention has been paid to food chains incorporating three
species in discrete systems \cite{Elalim2012,Azmy2008,zhang2009}.
However, few studies have analysed the full richness of the
parameter space analytically, where a diverse range of qualitative
dynamical regimes exist. In this paper we address this problem by
using a simple trophic model of three species interactions that
generalises a previous two-dimensional predator-prey model, given by
the difference Equations~(4.5) in \cite{Holden1986} (see also
\cite{Chapter_chaos}). The two-dimensional model assumes a food chain
structure with an upper limit to the total population of preys, whose
growth rate is affected by a single predator. The new
three-dimensional model explored in this article introduces a new top
predator species that consumes the predator and interferes in the
growth of the preys.

We provide a full description of the local dynamics and the
bifurcations in a wide region of the three-dimensional parameter
space containing relevant ecological dynamics. This parameter cuboid
is built using the prey's growth rates and the two predation rates as
axes. The first predation rate concerns to the predator that consumes
the preys, while the second predator rate is the consumption of the
first predator species by the top predator. As we will show, this
model displays remarkable examples of strange chaotic attractors. The
route to chaos associated to increasing predation strengths are shown
to be given by period-doubling bifurcations of invariant curves,
which arise via a supercritical Neimark-Sacker bifurcation.

\section{Three species predator-prey map}\label{se:themodel}
Discrete-time dynamical systems are appropriate for describing the
population dynamics of species with non-overlapping generations
\cite{Constantino1997,May1974,May1976,Kon2001,Aalberg2012}.
Such species\\[1pt]
\begin{minipage}{0.7\textwidth}
are found in temperate and boreal regions because of their seasonal
environments.
We here consider a food chain of three interacting species,
each with non-overlapping generations,
which undergoes intra-specific competition.
We specifically consider a population of preys $x$ which is predated
by a first predator with population $y.$
We also consider a third species given by a top predator $z$ that
predates on the first predator $y$, also interfering in the growth of prey's
population according to the side diagram.

Examples of top-predator$\to$predator$\to$prey interactions in
\end{minipage}\hfill\begin{minipage}{0.25\textwidth}
\includegraphics[width=\textwidth]{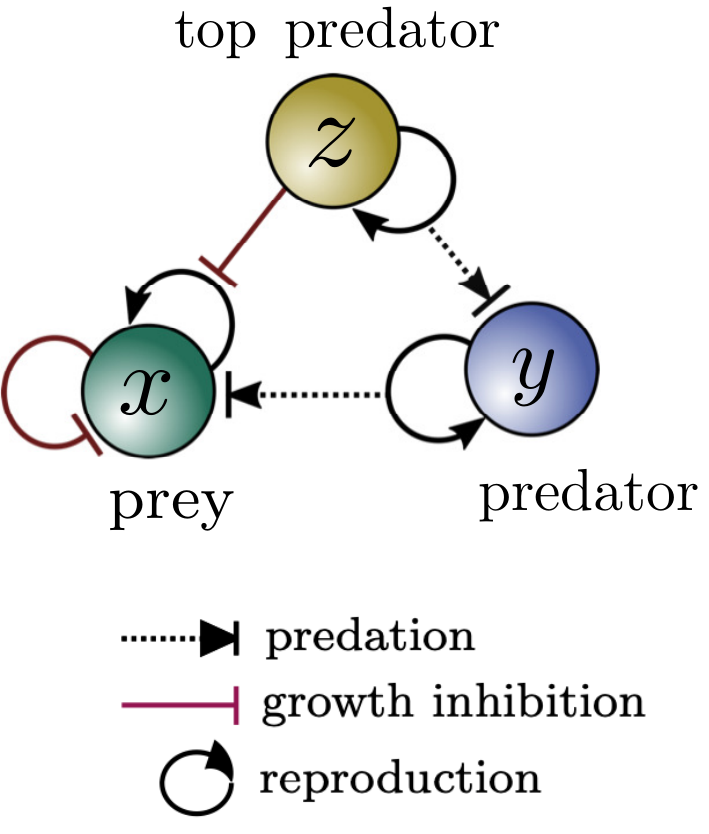}
\end{minipage}\vspace*{0.4ex}
univoltine populations can be found in ecosystems.
For instance, the heteroptera species \emph{Picromerus bidens}
in northern Scandinavia \cite{Saulich2014},
which predates the butterfly \emph{Pararge aegeria} by consuming on its eggs.
Also, other species such as spiders can act as top-predators
(e.g., genus \emph{Clubiona sp.}, with a wide distribution in northern Europe and Greenland).
The proposed model to study such ecological interactions can be described by
the following system of nonlinear difference equations:
\begin{equation}\label{eq:sistema}
\begin{pmatrix}
  x_{n+1} \\ y_{n+1} \\ z_{n+1}
\end{pmatrix} = T \begin{pmatrix}
  x_{n} \\ y_{n} \\ z_{n}
\end{pmatrix}
\phantom{x}  \text{where} \phantom{x}
T \begin{pmatrix} x \\ y \\ z \end{pmatrix} = \begin{pmatrix}
  \mu x ( 1 - x - y - z) \\
  \beta y (x - z) \\
  \gamma y z
\end{pmatrix}
\end{equation}
and $x,y,z$ denote population densities with respect to a normalized
carrying capacity for preys ($K = 1$).
Observe that, in fact, if we do not normalize the carrying capacity
the term $1 - x - y$ in $T_{\mu,\beta}$ should read
$\tfrac{1 - x}{K - y}.$
Constants $\mu,\beta,\gamma$ are positive.
In the absence of predation, as mentioned, preys grow logistically
with an intrinsic reproduction rate $\mu.$
However, preys' reproduction is decreased by the action of predation
from both predators $y$ and $z.$
Parameter $\beta$ is the growth rate of predators $y$, which is
proportional to the consumption of preys.
Finally, $\gamma$ is the growth rate of predators $z$ due to the
consumption of species $y.$
Notice that predator $z$ also predates (interferes) on $x$, but it is
assumed that the increase in reproduction of the top predator $z$
is mainly given by the consumption of species $y.$

Model~\eqref{eq:sistema} is defined on the phase space,
given by the simplex
\[
   \U := \bigl\{(x,y,z)\in \R^3\, \colon x,y,z \geq 0 \text{ and }
                x+y+z \leq 1\bigr\}
\]
and, although it is meaningful for the parameters' set
\[
   \{(\mu,\beta,\gamma)\in \R^3 \, \colon \mu > 0,\ \beta > 0
   \text{ and } \gamma > 0\}
\]
of all positive parameters, we will restrict ourselves to the
following particular cuboid
\begin{equation}\label{def:cuboid:direct}
  \Q = \left\{ (\mu,\beta,\gamma) \in (0,4] \times [2.5,5] \times [5,9.4]\right\}
\end{equation}
which exhibits relevant biological dynamics
(in particular bifurcations and routes to chaos).

The next proposition lists some very simple dynamical facts about
System~\eqref{eq:sistema} on the domain $\U$ with parameters in the
cuboid $\Q.$
It is a first approximation to the understanding of the dynamics of
this system.

A set $A\subset \U$ is called \emph{$T$-invariant} whenever
$T(A) \subset A.$

\begin{proposition}\label{prop:simpledynamics}
The following statements hold for System~\eqref{eq:sistema} and all
parameters $(\mu,\beta,\gamma) \in \Q.$
\begin{enumerate}[(a)]
\item The point $(0,0,0) \in \U$ is a fixed point of $T$ which
corresponds to extinction of the three species.
\item $
   T\bigl(\{(1,0,0)\}\bigr) = T\bigl(\{(0,y,0)\in\U\}\bigr) =
   T\bigl(\{(0,0,z)\in\U\}\bigr) = (0,0,0).
$
That is, the point $(1,0,0)$ and every initial condition in $\U$ on
the $y$ and $z$ axes lead to extinction in one iterate.
\item $
   T\bigl(\{(x,0,z)\in\U\}\bigr) \subset \{(x,0,0)\in\U\} \subset
   \{(x,0,z)\in\U\}.
$
In particular the sets $\{(x,0,z)\in\U\}$ and $\{(x,0,0)\in\U\}$ are
$T$-invariant.
\end{enumerate}
\end{proposition}

\begin{proof}
Statements (a) and (b) follow straightforwardly.
To prove (c) notice that
$T\bigl((x,0,z)\bigr) = (\mu x(1-x-z),0,0)$ with $\mu \in (0,4],$
$x \geq 0$ and $x+z \leq 1.$
Hence,
\[
   0 \leq \mu x(1-x-z) = \mu x(1-x) - \mu xz \leq 1 - \mu xz \leq 1,
\]
and thus $(\mu x(1-x-z),0,0) \in \U.$
\end{proof}

An important natural question is: what is the (maximal) subset $\SD$
of $\U$ where the Dynamical System associated to Model~\eqref{eq:sistema} is
well defined for all times or iterates
(i.e. $T^n\bigl((x,y,z)\bigr)\in\U$ for every $n \in \N$ and
$(x,y,z)\in\SD$).
Such a set is called the \emph{dynamical domain} or
the \emph{invariant set} of System~\eqref{eq:sistema}.
The domain $\SD$ is at the same time complicate and difficult to
characterize (see Figure~\ref{fig:TallConjuntInvariant}).
\begin{figure}[b]
\begin{center}
\hfill\subfigdef[0.235]{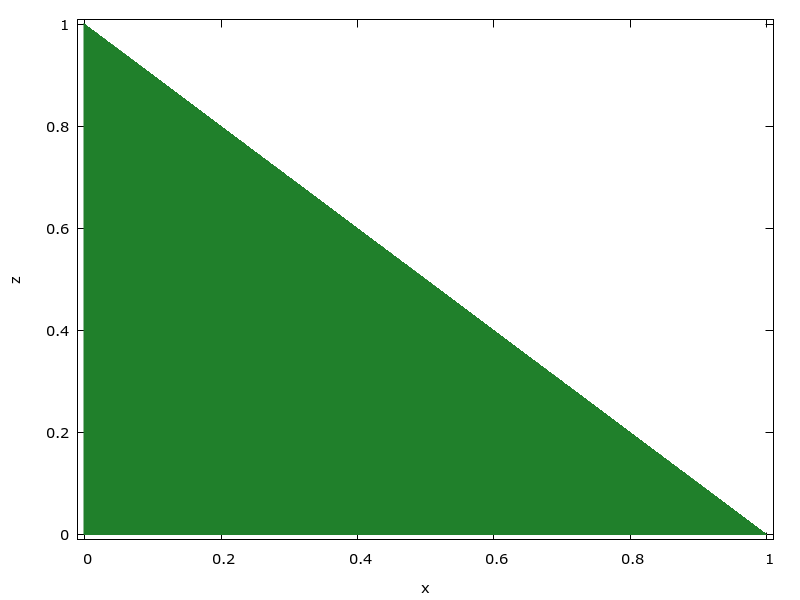}{$\mu = 0.7,\ \beta = 2.5$\newline and $\gamma = 5.0$:\newline Plane $y = 0.$}
\hfill\subfigdef[0.235]{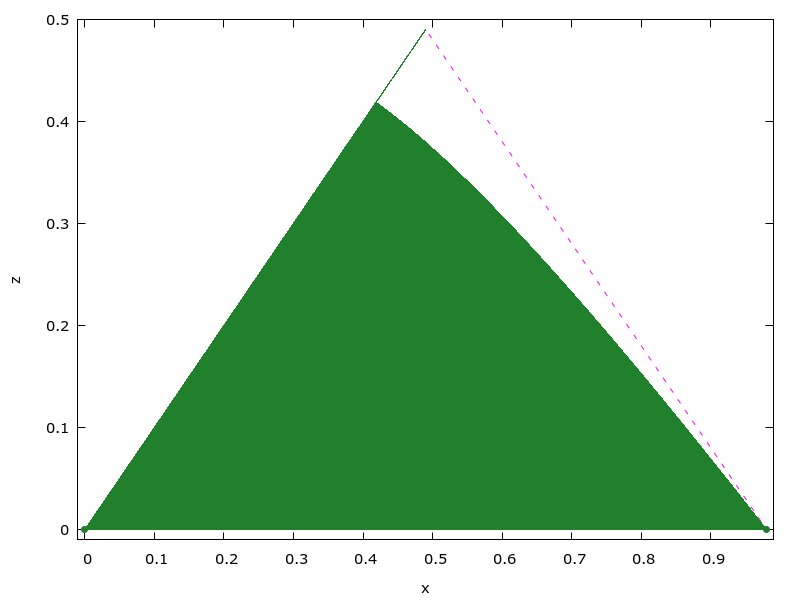}{$\mu = 0.7,\ \beta = 2.5$\newline and $\gamma = 5.0$:\newline Plane $y = 0.02.$}
\hfill\subfigdef[0.235]{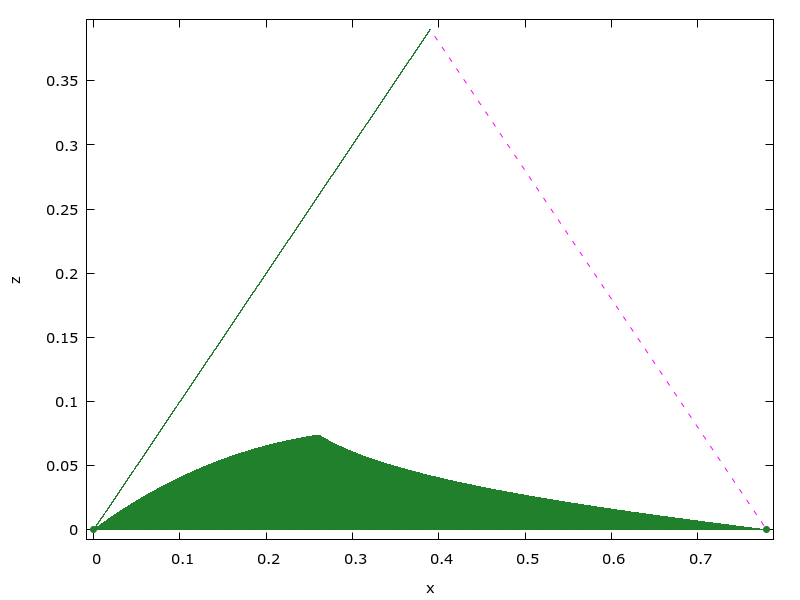}{$\mu = 0.7,\ \beta = 2.5$\newline and $\gamma = 5.0$:\newline Plane $y = 0.22.$}
\hfill\subfigdef[0.235]{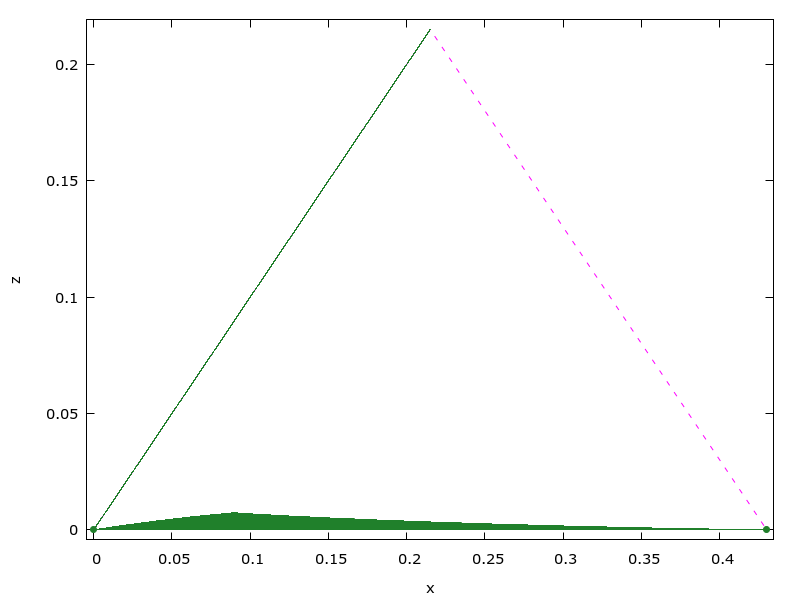}{$\mu = 0.7,\ \beta = 2.5$\newline and $\gamma = 5.0$:\newline Plane $y = 0.57.$}
\hfill \strut \\[-2ex]
\hfill\subfigdef[0.235]{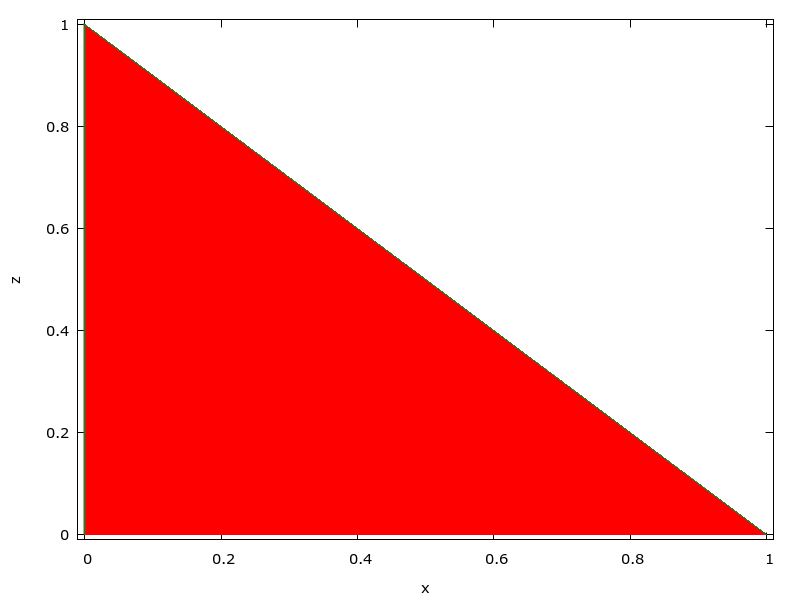}{$\mu = 1.261,$\newline$\beta = 2.925$\newline$\gamma = 5.748$:\newline Plane $y = 0.$}
\hfill\subfigdef[0.235]{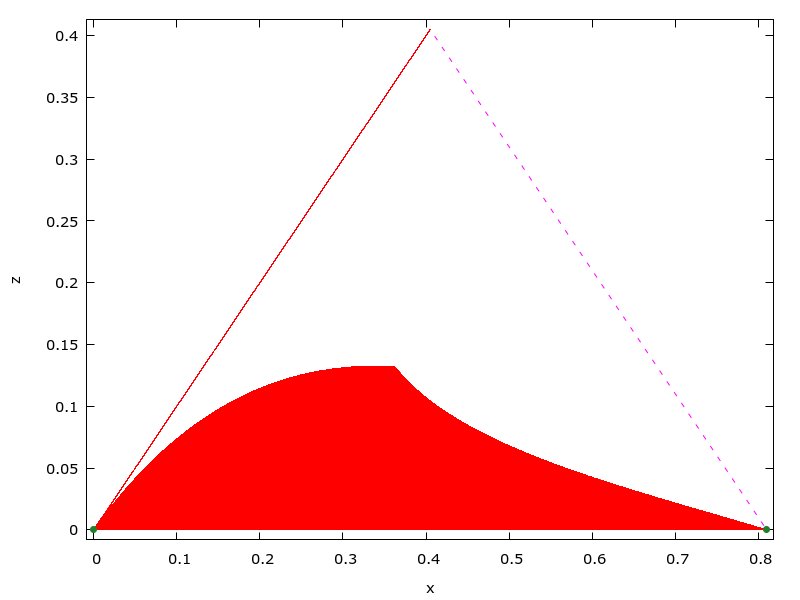}{$\mu = 1.261,$\newline$\beta = 2.925$\newline$\gamma = 5.748$:\newline Plane $y = 0.19.$}
\hfill\subfigdef[0.235]{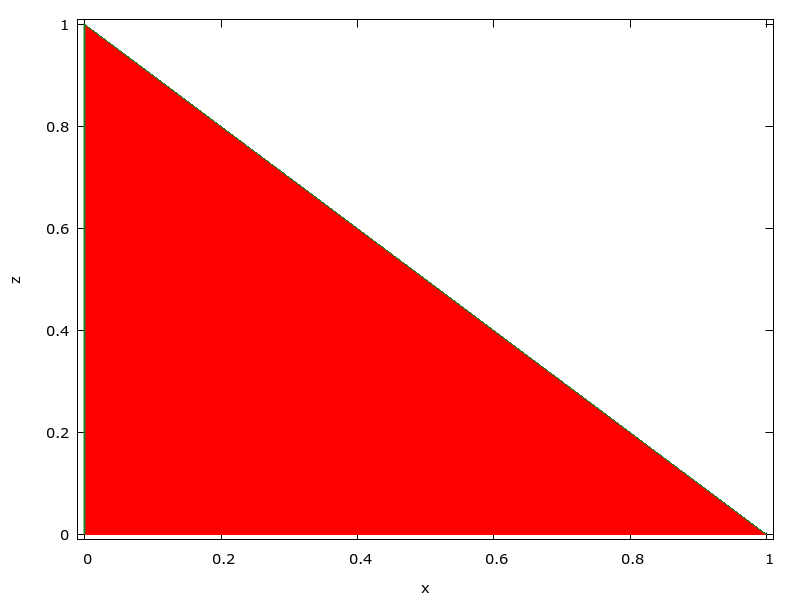}{$\mu = 1.657,$\newline$\beta = 3.225$\newline$\gamma = 6.276$:\newline Plane $y = 0.$}
\hfill\subfigdef[0.235]{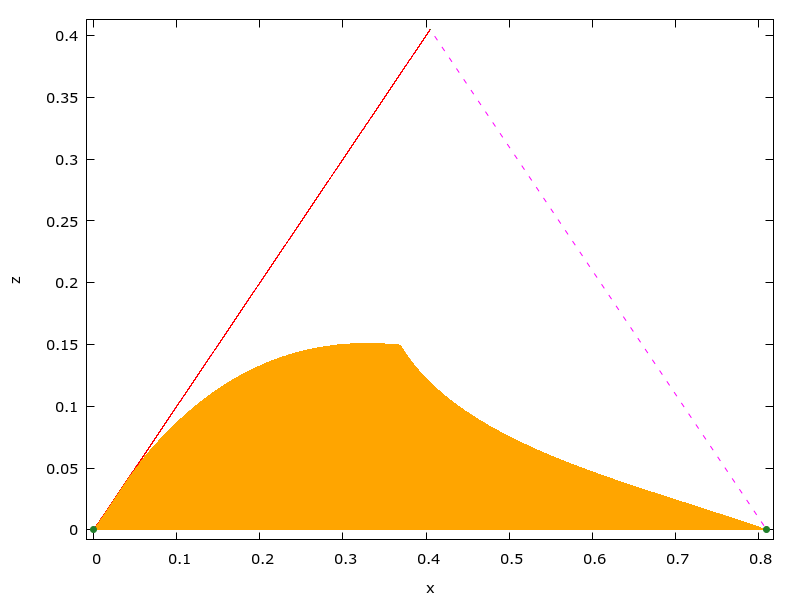}{$\mu = 1.657,$\newline$\beta = 3.225$\newline$\gamma = 6.276$:\newline Plane $y = 0.19.$}
\hfill \strut \\[-2ex]
\hfill\subfigdef[0.235]{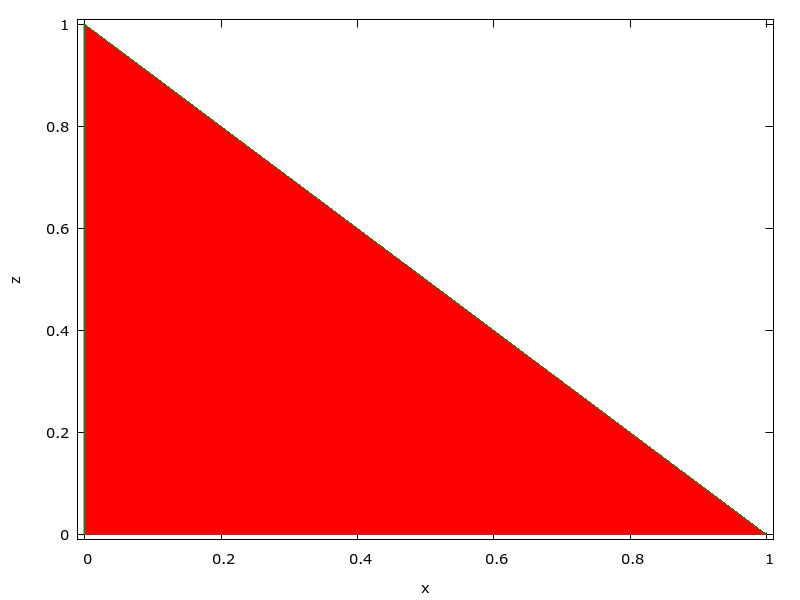}{$\mu = 1.822,$\newline$\beta = 3.350$\newline$\gamma = 6.496$:\newline Plane $y = 0.$}
\hfill\subfigdef[0.235]{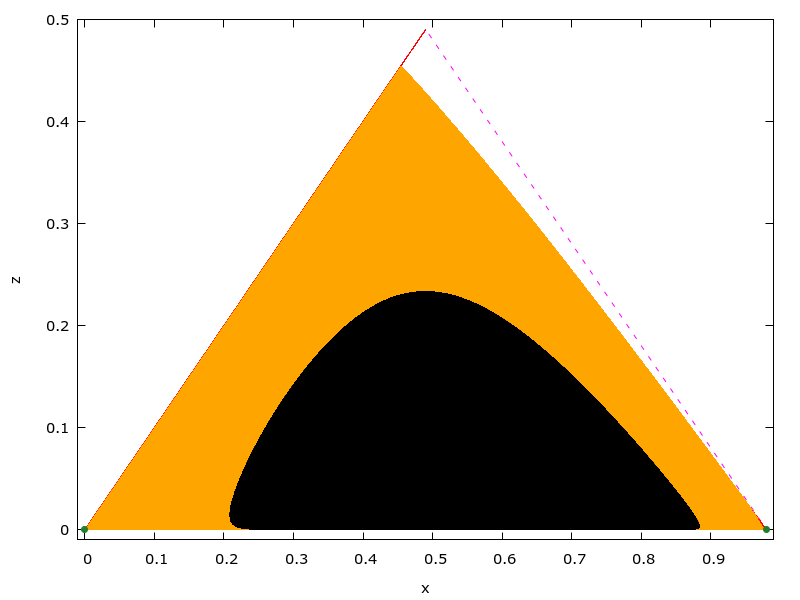}{$\mu = 1.822,$\newline$\beta = 3.350$\newline$\gamma = 6.496$:\newline Plane $y = 0.02.$}
\hfill\subfigdef[0.235]{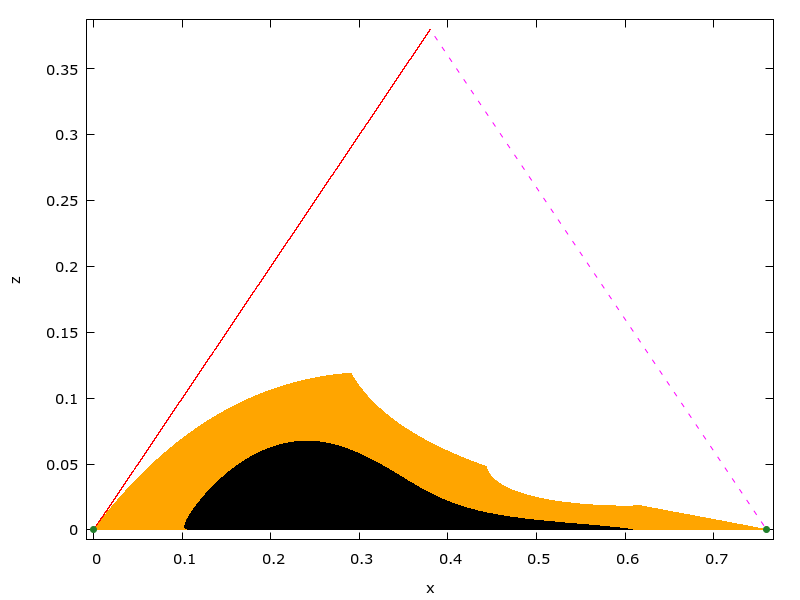}{$\mu = 1.822,$\newline$\beta = 3.350$\newline$\gamma = 6.496$:\newline Plane $y = 0.24.$}
\hfill\subfigdef[0.235]{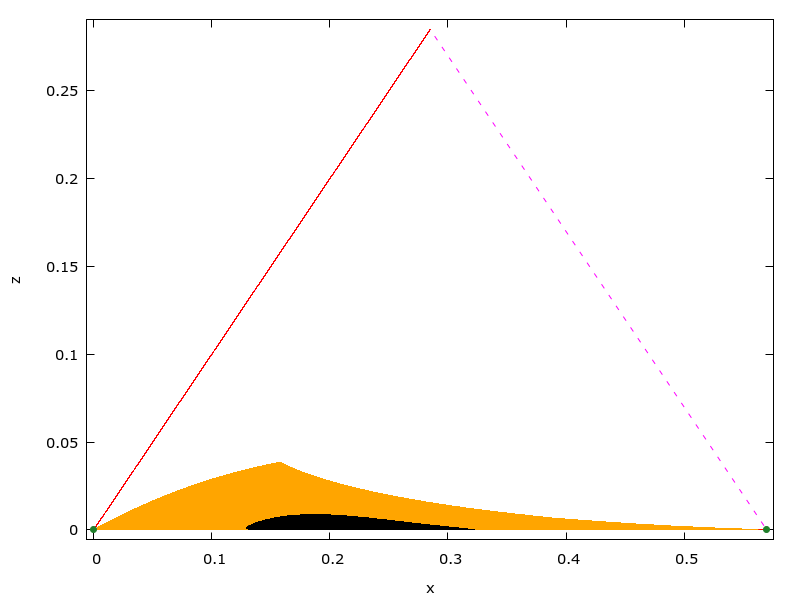}{$\mu = 1.822,$\newline$\beta = 3.350$\newline$\gamma = 6.496$:\newline Plane $y = 0.43.$}
\hfill \strut \\[-2ex]
\hfill\subfigdef[0.235]{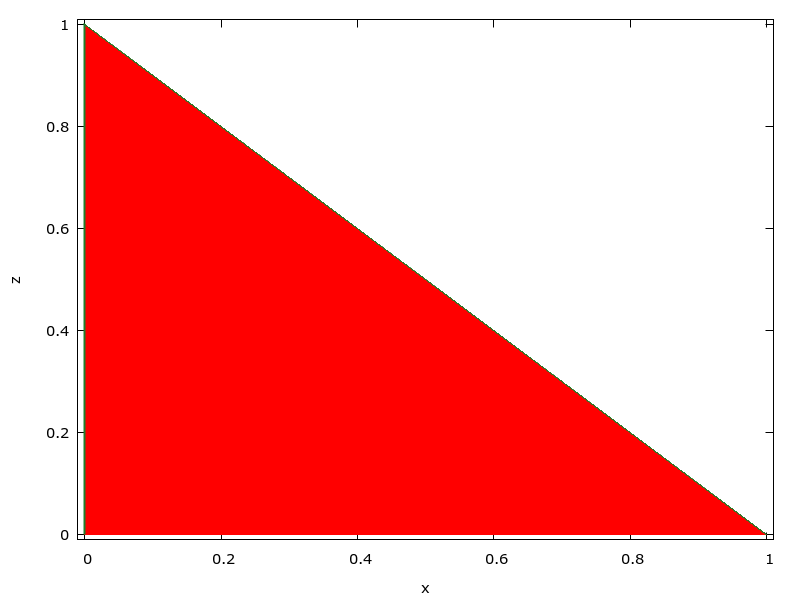}{$\mu = 2.218,$\newline$\beta = 3.65,$\newline$\gamma = 7.024$:\newline Plane $y = 0.$}
\hfill\subfigdef[0.235]{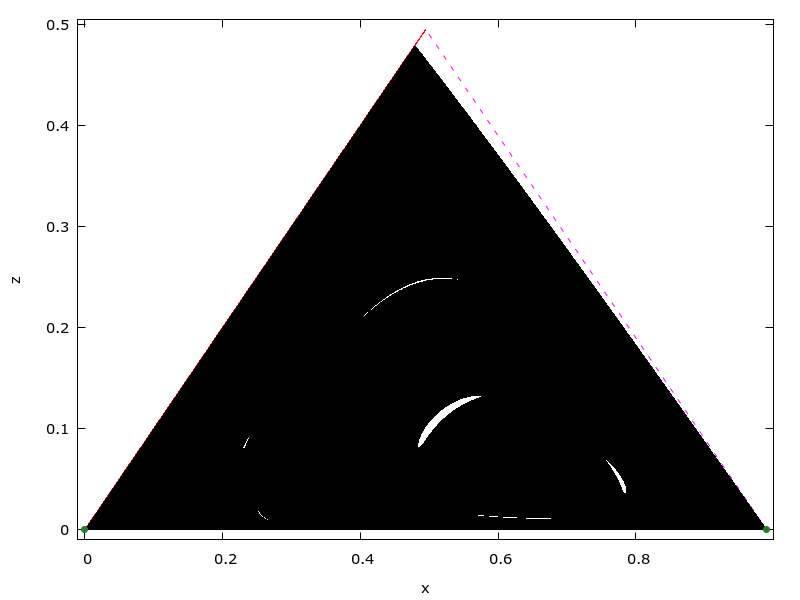}{$\mu = 2.218,$\newline$\beta = 3.65,$\newline$\gamma = 7.024$:\newline Plane $y = 0.01.$}
\hfill\subfigdef[0.235]{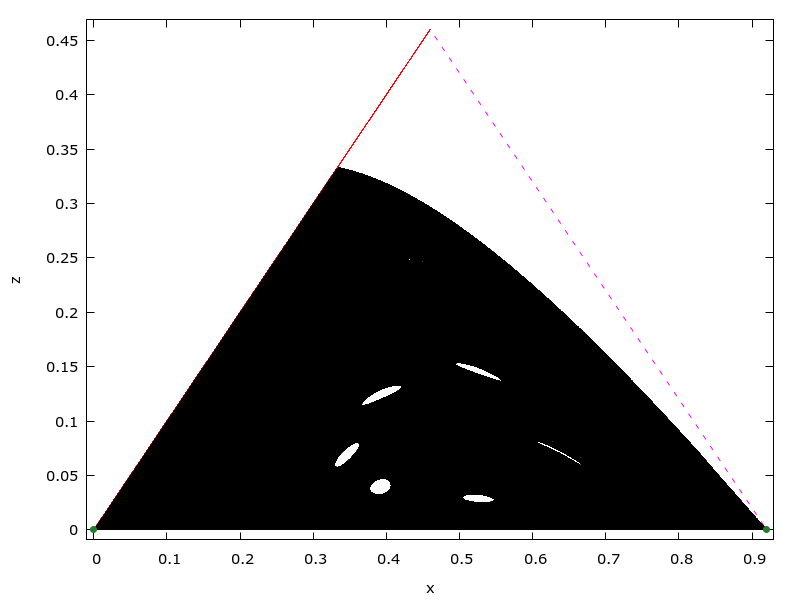}{$\mu = 2.218,$\newline$\beta = 3.65,$\newline$\gamma = 7.024$:\newline Plane $y = 0.08.$}
\hfill\subfigdef[0.235]{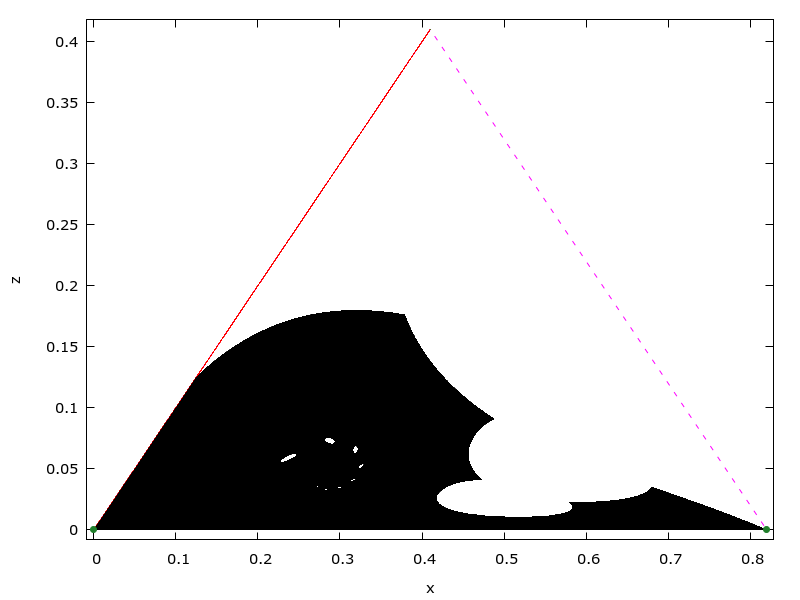}{$\mu = 2.218,$\newline$\beta = 3.65,$\newline$\gamma = 7.024$:\newline Plane $y = 0.18.$}
\hfill \strut \\[-2ex]
\hfill\subfigdef[0.235]{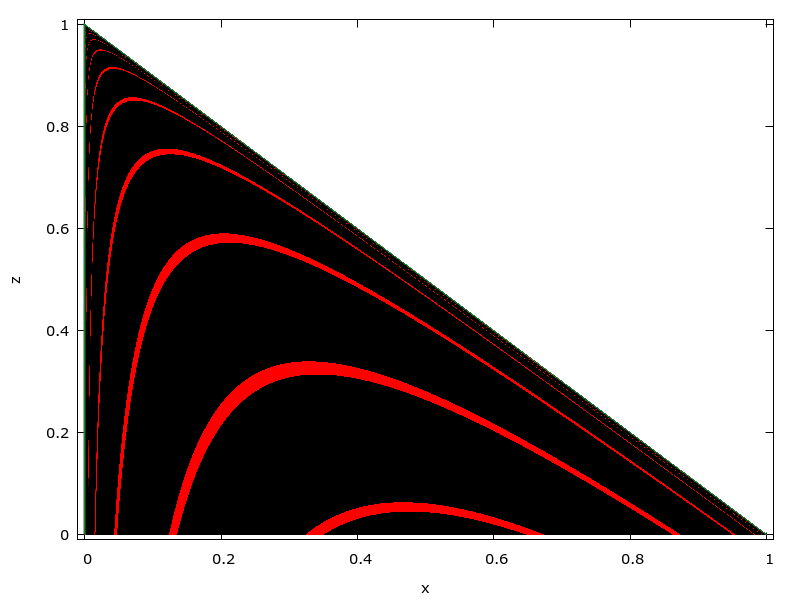}{$\mu = 2.977,$\newline$\beta = 4.225,$\newline$\gamma = 8.036$:\newline Plane $y = 0.$}
\hfill\subfigdef[0.235]{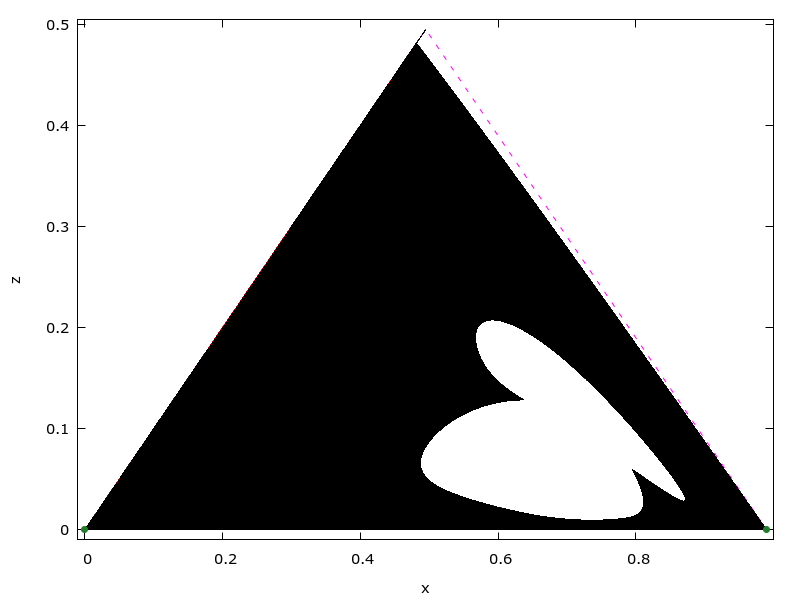}{$\mu = 2.977,$\newline$\beta = 4.225,$\newline$\gamma = 8.036$:\newline Plane $y = 0.01.$}
\hfill\subfigdef[0.235]{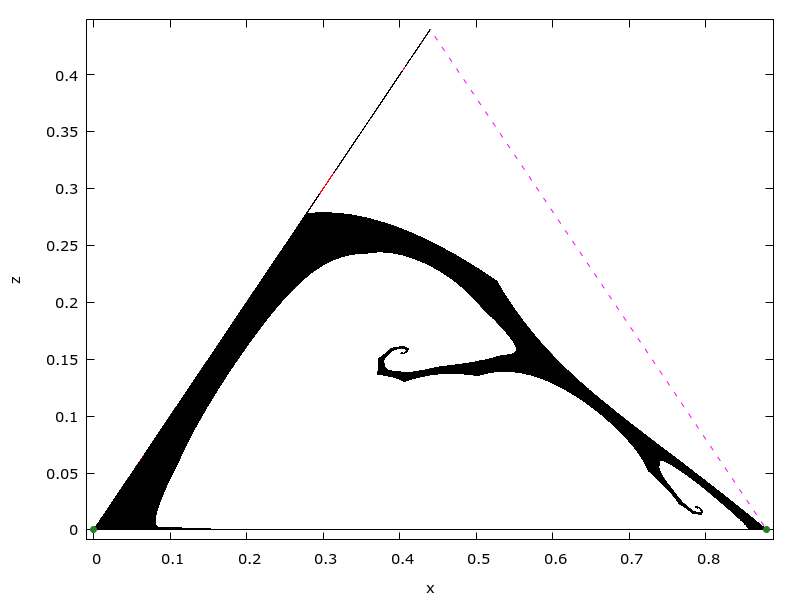}{$\mu = 2.977,$\newline$\beta = 4.225,$\newline$\gamma = 8.036$:\newline Plane $y = 0.12.$}
\hfill\subfigdef[0.235]{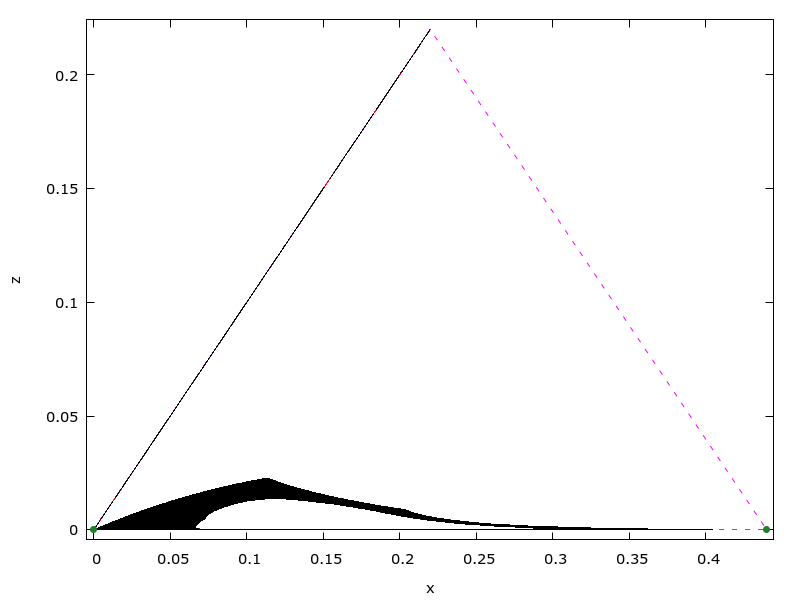}{$\mu = 2.977,$\newline$\beta = 4.225,$\newline$\gamma = 8.036$:\newline Plane $y = 0.56.$}
\hfill \strut
\end{center}
\captionsetup{width=\linewidth}
\caption{Plots of the intersection of $\SD$ with the planes $y =
\text{ctnt}$ for several choices of the planes and the parameters
$\mu,$ $\beta$ and $\gamma.$
Points drawn in \textcolor{DarkGreen}{dark green} color converge
to the fixed point \textcolor{DarkGreen}{$(0,0,0)$},
points in \textcolor{red}{red} converge to the fixed point
\textcolor{red}{$\left(\tfrac{\mu-1}{\mu},0,0\right)$},
points in \textcolor{orange}{orange} converge to the fixed point
\textcolor{orange}{$\left(\beta^{-1}, 1 - \mu^{-1} - \beta^{-1}, 0\right)$},
and points in black belong to the invariant set $\SD$ but
do not belong to the basin of attraction of any fixed point.
The dashed \textcolor{magenta}{magenta} lines show the boundary
of the cut of the plane $y = \text{ctnt}$ with the domain $\RD.$}\label{fig:TallConjuntInvariant}
\end{figure}

To get a, perhaps simpler, definition of the \emph{dynamical domain} $\SD$
we introduce the \emph{one-step escaping set}
$\Theta = \Theta(\mu,\beta,\gamma)$ of System~\eqref{eq:sistema}
defined as the set of points
$(x,y,z) \in \U$ such that $T\bigl((x,y,z)\bigr) \notin \U,$
and the \emph{escaping set} $\Gamma = \Gamma(\mu,\beta,\gamma)$
as the set of points $(x,y,z) \in \U$ such that
$T^n\bigl((x,y,z)\bigr) \notin \U$
for some $n\geq 1.$
Clearly,
\[
  \Gamma = \U \cap \left(\bigcup_{n=0}^{\infty} T^{-n}\bigl(\Theta\bigr)\right).
\]
\begin{figure}
\centering
\includegraphics[width=\textwidth]{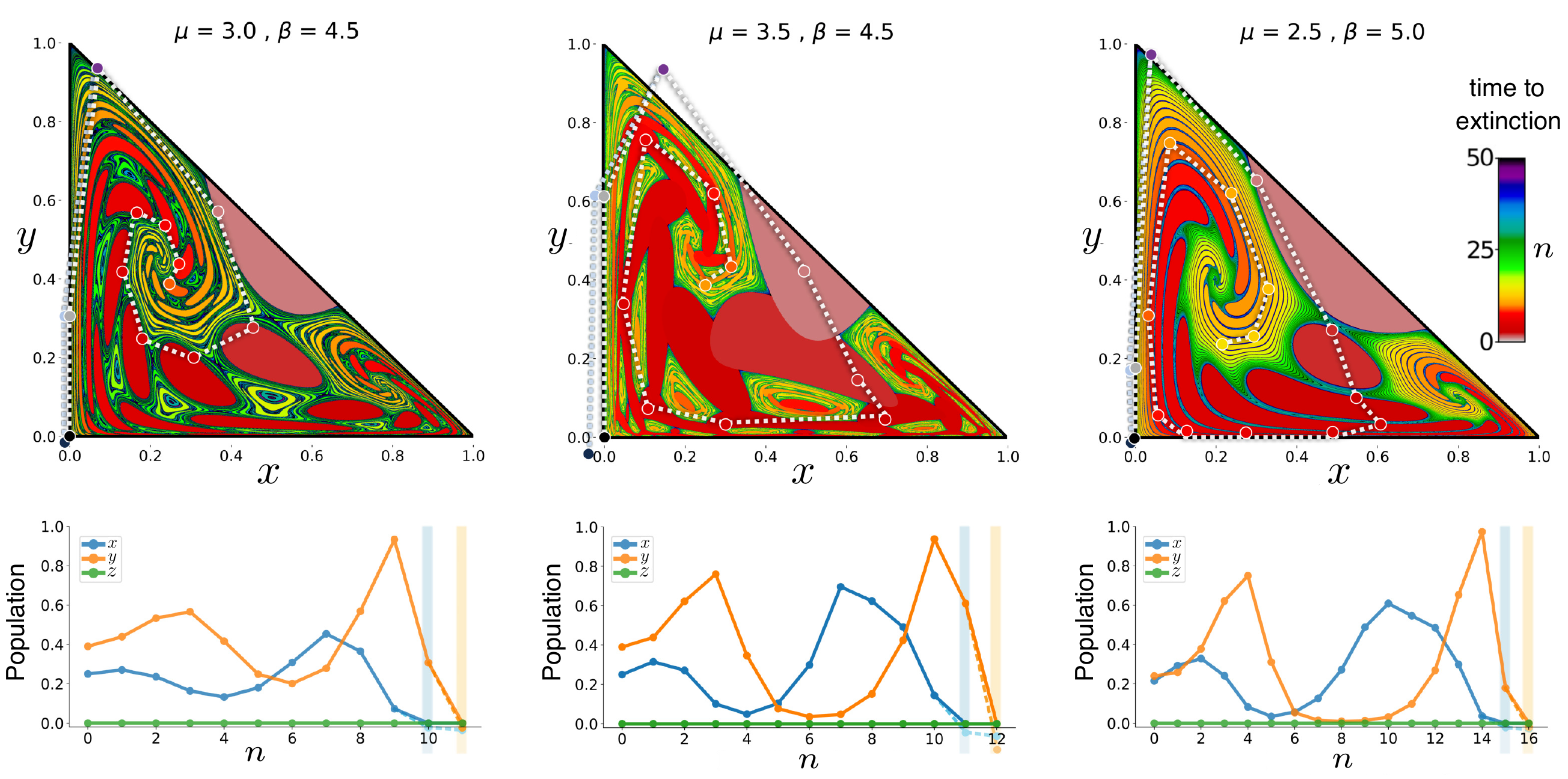}
\captionsetup{width=\linewidth}
\caption{(Upper) Escaping sets, $\Gamma$, obtained by iteration
for parameter values giving place to complex (fractal) shapes,
computed on the phase plane $z = 0.$
The colours display the time that a given orbit overcomes
the carrying capacity then going to extinction
(colour gradient indicates the number of iterations to extinction,
from $1$ (pink) to $50$ (violet) iterations).
The small circles connected by the dashed white line indicate
how iterates move from the initial condition
$(x_0, y_0, z_0) = (0.25, 0.39, 0)$ towards extinction.
(Bottom) Extinction time series for preys and predators, $y.$
From left to right:
$(x_0, y_0, z_0) = (0.25, 0.39, 0)$ and
$(\mu, \beta, \gamma) = (3.0, 4.5, 7.5);$
$(x_0, y_0, z_0) = (0.25,0.39,0.00)$ and
$(\mu,\beta,\gamma) = (3.5,4.5,7.5);$ and 
$(x_0, y_0, z_0) = (0.215, 0.24, 0)$ and
$(\mu,\beta,\gamma) = (2.5,5.0,7.5).$
The vertical bars indicate when iterates for $x$ and $y$
become negative after overcoming the carrying capacity.}\label{fig:Extinctions}
\end{figure}

Several examples of these escaping sets are displayed in the
first row of Figure~\ref{fig:Extinctions} for parameter values
giving place to complex (apparently fractal) sets.
Specifically, the shown escaping sets are coloured by
the number of iterates (from $1$ to $50$) needed to
leave the domain $\SD$ (sum of the populations above
the carrying capacity), after which populations jump
to negative values (extinction),
involving a catastrophic extinction.
The associated time series are displayed below each panel
in the second row in Figure~\ref{fig:Extinctions}.
After an irregular dynamics the prey and predator
populations become suddenly extinct,
as indicated by the vertical rectangles at the end
of the time series.
We want to emphasise that these extinctions are due
to the discrete nature of time.
That is, they have nothing to do with the $\omega$-limits
of the dynamical system that are found within and at the
borders of the simplex.
For the sake of clarity, these results are illustrated
setting the initial number of top predators to $z_0 = 0$
(see also \textsf{Movie-1.mp4} for an animation on how
the escaping sets change depending on the parameters).
However, similar phenomena are found in the full simplex
with an initial presence of all of the species.

The \emph{dynamical domain} or \emph{invariant set} of
System~\eqref{eq:sistema} can also be defined as:
\[
  \SD := \U \setminus \Gamma = \U \setminus \bigcup_{n=0}^{\infty} T^{-n}\bigl(\Theta\bigr).
\]
In other words, the initial conditions that do not belong to the
dynamical domain $\SD$ are, precisely, those that belong to the
escaping set $\Gamma$ which consists of those initial conditions in $\U$
that stay in $\U$ for some iterates (and hence are well defined),
and finally leave $\U$ in a catastrophic extinction that cannot be
iterated (System~\eqref{eq:sistema} is not defined on it).

In general we have $\Gamma \supset \Theta \neq \emptyset$ and, hence,
$\SD \varsubsetneq \U$
(that is, $\U$ may not be the dynamical domain of
System~\eqref{eq:sistema}).
On the other hand, for every $\mu,\beta,\gamma > 0,$
$\SD$ is non-empty (it contains at least the point $(0,0,0)$) and
$T$-invariant.
Hence
\[
   \SD = \bigl\{T^n\bigl((x,y,z)\bigr)\, \colon (x, y, z)\in\SD
                \text{ and } n\in\N\cup\{0\}\bigr\}.
\]
Moreover, since the map $T$ is (clearly) non-invertible,
a backward orbit of a point from $\SD$ is not uniquely defined.

As we have pointed out, the domain $\SD$ is at the same time
complicate and difficult to characterize.
However, despite of the fact that this knowledge is important for the
understanding  of the global dynamics, in this paper we will
omit this challenging study and we will consider
System~\eqref{eq:sistema} on the domain
\[
    \RD = \{(x,0,z)\in\U\} \cup \{(x,y,z)\in\U \, \colon y > 0
            \text{ and } x \geq z\}.
\]
(see Figure~\ref{fig:domain-walledsimplex})
which is an approximation of $\SD$ better than $\U,$ as stated in the
next proposition.
\begin{figure}[ht]
\begin{tikzpicture}
\begin{axis}[grid=major, view={215}{20}, width=300pt, height=200pt,
             xmin=0, xmax=1, ymin=0, ymax=1, zmin=0, zmax=1,
             xlabel=$x$, ylabel=$y$, ytick={0, 0.2, 0.4, 0.6, 0.8, 1},
             zlabel=$z$, ztick={0, 0.2, 0.4, 0.6, 0.8, 1},
             enlargelimits=0.025]
\addplot3[draw=blue, ultra thick, fill=blue, fill opacity=0.3] coordinates {(0,0,0) (0,0,1) (0.99,0,0)} \closedcycle; 
\addplot3[draw=none, fill=olive, fill opacity=0.7] coordinates {(0,0,0) (0,1,0) (1,0,0)} \closedcycle; 
\addplot3[draw=none, fill=magenta, fill opacity=0.3] coordinates {(1,0,0) (0.5,0,0.5) (0,1,0)} \closedcycle; 
\addplot3[color=magenta, ultra thick] coordinates {(-0.01,1,0) (0.999,0,0)};
\addplot3[color=black, ultra thick] coordinates {(0.5,0,0.5) (0,0.9925,0)};
\addplot3[draw=none, fill=gray, fill opacity=0.3] coordinates {(0,0,0) (0.5,0,0.5) (0,1,0)} \closedcycle; 
\addplot3[color=gray, ultra thick] coordinates {(0,1,0) (0,0,0)};
\addplot3[color=blue] coordinates {(0,0,0) (0.5,0,0.5)};
\end{axis}
\end{tikzpicture}
\captionsetup{width=\linewidth}
\caption{Plot of the domain $\RD$.
The \textcolor{blue}{``wall'' $y=0$}: \textcolor{blue}{$\{(x,0,z)\in\U\}$} is drawn in \textcolor{blue}{blue}.
The \textcolor{DarkGreen}{face  $z=0$}:
\textcolor{DarkGreen}{$\{(x,y,0)\in\U\}$} is drawn in \textcolor{DarkGreen}{olive} color;
The \textcolor{magenta}{face  $x+y+z=1$}:
\textcolor{magenta}{$\{(x,y,z)\in\U \, \colon y>0,\ x \geq z \text{ and } x+y+z = 1\}$} in \textcolor{magenta}{magenta},
and the \textcolor{darkgray}{face  $x=z$}:
\textcolor{darkgray}{$\{(x,y,x)\in\U \, \colon y>0\}$} in \textcolor{darkgray}{gray}.}\label{fig:domain-walledsimplex}
\end{figure}
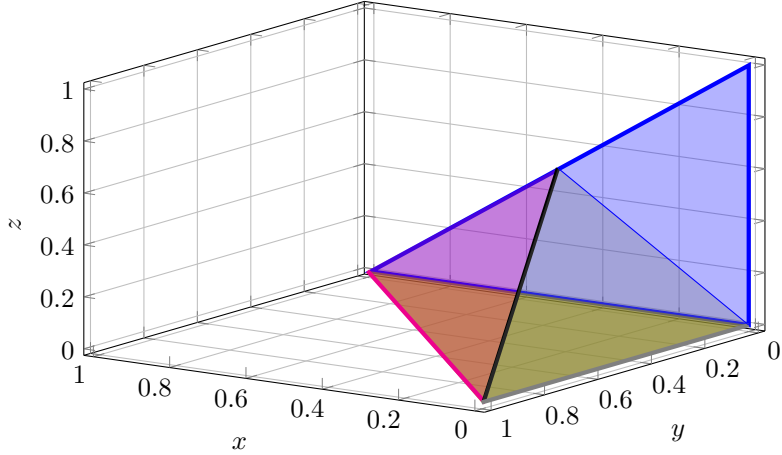

\begin{proposition}\label{prop:domain-walledsimplex}
For System~\eqref{eq:sistema} and all parameters $(\mu,\beta,\gamma) \in \Q$
we have
\[
    \{(x,0,z)\in\U\} \cup \{(0,y,0)\in\U\} \subset \SD \subset \RD.
\]
\end{proposition}

\begin{proof}
The fact that 
$\{(x,0,z)\in\U\} \cup \{(0,y,0)\in\U\} \subset \SD$
follows directly from Proposition~\ref{prop:simpledynamics}.
To prove the other inclusion observe that
\[
   \RD = \U \setminus \{(x,y,z)\in\U \, \colon y > 0 \text{ and } z > x\},
\]
and for every $(x,y,z)\in\U$ with $y > 0$ and $z > x$ we have
$\beta y (x - z) < 0$ because $\beta > 0.$
Consequently,
$
  \{(x,y,z)\in\U \, \colon y > 0 \text{ and } z > x\} \subset \Theta,
$
and hence,
\[
\SD = \U \setminus \bigcup_{n=0}^{\infty} T^{-n}\bigl(\Theta\bigr) 
    \subset \U \setminus \Theta
    \subset \U \setminus \{(x,y,z)\in\U \, \colon y > 0 \text{ and } z > x\} = \RD.
\]
\end{proof}

\section{Fixed points and local stability}
This section is devoted to compute the biologically-meaningful fixed
points of $T$ in $\RD,$ 
and to analyse their local stability.
This study will be carried out in terms of the positive parameters
$\mu,\beta,\gamma.$

The dynamical system defined by~\eqref{eq:sistema} has the following
four (biologically meaningful) fixed points in the domain $\RD$
(see Figure~\ref{fig:EntradesISortides}):
\begin{align*}
  P_1^* &= (0, 0, 0), \\
  P_2^* &= \left(\frac{\mu - 1}{\mu}, 0, 0 \right),\\
  P_3^* &= \left(\frac{1}{\beta}, 1 - \frac{1}{\mu} - \frac{1}{\beta}, 0 \right), \\
  P_4^* &= \left(
      \frac{1}{2} \left( 1 - \mu^{-1} + \beta^{-1} - \gamma^{-1} \right),
      \frac{1}{\gamma}, 
      \frac{1}{2} \left( 1 - \mu^{-1} - \beta^{-1} - \gamma^{-1} \right) \right).
\end{align*}
Notice that the system admits a fifth fixed point
$P_5^* = \left(0, \tfrac{1}{\gamma}, -\tfrac{1}{\beta}\right),$
which is not biologically meaningful since it has a negative
coordinate (recall that $\beta > 0$), and thus it will not be taken
into account in this study.

\begin{figure}[ht]
\centering
\includegraphics[width=\textwidth]{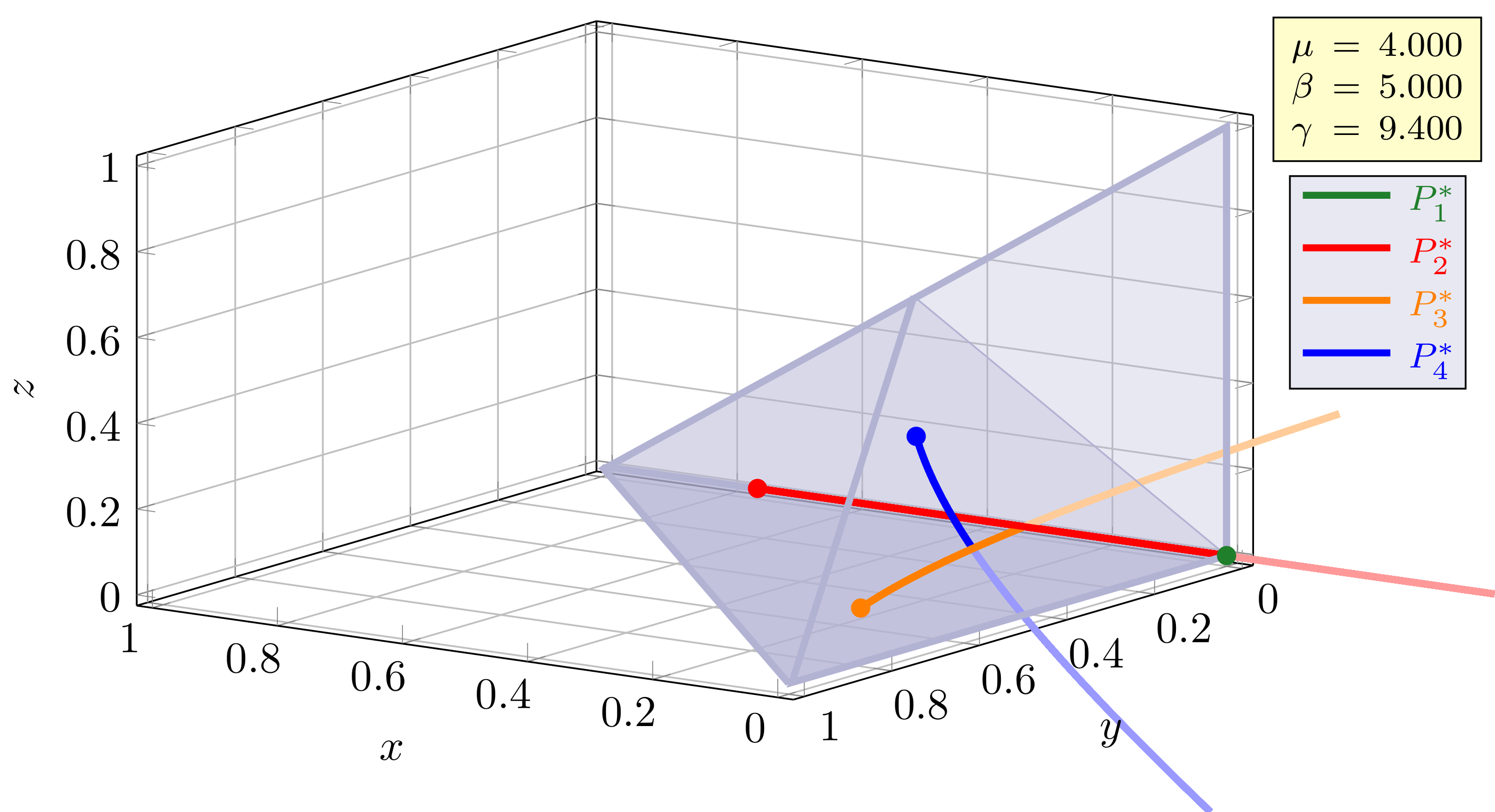}
\captionsetup{width=\linewidth}
\caption{The fixed point 
\textcolor{DarkGreen}{$P^*_1$} in \textcolor{DarkGreen}{dark green}
color and the paths described by
\textcolor{red}{$P^*_2$} in \textcolor{red}{red},
\textcolor{orange}{$P^*_3$} in \textcolor{orange}{orange} and
\textcolor{blue}{$P^*_4$} in \textcolor{blue}{blue}
in the domain 
\textcolor{ColorMapViolet!70!blue}{$\RD$} 
(plotted in \textcolor{ColorMapViolet!70!blue}{violet})
when the parameters $(\mu,\beta,\gamma)$ follow the path
$
   (\mu(t),\beta(t),\gamma(t)) = (3.3,2.5,4.4)t + (0.7,2.5,5)
$
with $t$ ranging from 0 to 1.
The pieces of the paths outside the domain
\textcolor{ColorMapViolet!70!blue}{$\RD$},
which correspond to non biologically meaningful situations, are drawn
with the color softened.
The path described by the fixed point 
\textcolor{red}{$P^*_2$} in \textcolor{ColorMapViolet!70!blue}{$\RD$}
bifurcates from 
\textcolor{DarkGreen}{$P^*_1$} when $\mu = 1.$
The path described by 
\textcolor{orange}{$P^*_3$} in \textcolor{ColorMapViolet!70!blue}{$\RD$}
bifurcates from 
\textcolor{red}{$P_2^*$} when $\mu = \tfrac{\beta}{\beta -1}.$
The path of 
\textcolor{blue}{$P^*_4$} in \textcolor{ColorMapViolet!70!blue}{$\RD$}
bifurcates from 
\textcolor{orange}{$P_3^*$} when $\mu^{-1}+\beta^{-1}+\gamma^{-1} = 1$
(or, equivalently, when $\mu = \tfrac{\beta\gamma}{(\beta - 1)\gamma - \beta}$).}\label{fig:EntradesISortides}
\end{figure}

The fixed points
$P_1^*$, $P_2^*$, and $P_3^*$
are boundary equilibria, while $P_4^*$ is a boundary equilibrium if
$\mu^{-1}+\beta^{-1}+\gamma^{-1} = 1$
and interior otherwise.
The fixed point $P_1^*$ is the origin, representing the extinction of
all the species. $P_2^*$ is a boundary fixed point, with absence of
the two predator species. The point $P_3^*$ is the boundary fixed
point in the absence of the top-level predator $z=0$, while the
point $P_4^*$, when it is located in the interior of $\RD$,
corresponds to a coexistence equilibrium.

The next lemma gives necessary and sufficient conditions in order
that the fixed points $P_1^*, P_2^*, P_3^*$, and $P_4^*$ are
biologically meaningful (belong to the domain $\U$ and, hence, to
$\RD$ --- see, for instance, Figure~\ref{fig:EntradesISortides} and
\textsf{Movie-2.avi} in the Supplementary Material; see also the right
part of Figure~\ref{SuPeRfIgUrE}).

\begin{lemma}\label{lem:existence:fixpoints}
The following statements hold for every parameters' choice
$(\mu,\beta,\gamma)\in \Q$:
\begin{enumerate}[\boldmath $P_1^*:$]
\item The fixed point $P_1^*$ belongs to $\RD.$
\item The fixed point $P_2^*$ belongs to $\RD$ if and only if $\mu
\geq 1.$ Moreover, $P_2^* = P_1^*$ if and only if $\mu = 1.$
\item The fixed point $P_3^*$ belongs to $\RD$ if and only if
$\mu \geq \tfrac{\beta}{\beta-1} \geq \tfrac{5}{4}$
(which is equivalent to $\tfrac{1}{\mu} + \tfrac{1}{\beta} \leq 1$).
Moreover, $P_3^* = P_2^*$ if and only if $\mu = \tfrac{\beta}{\beta-1}.$
\item The fixed point $P_4^*$ belongs to $\RD$ if and only if
$\mu^{-1}+\beta^{-1}+\gamma^{-1} \leq 1$
(which is equivalent to $\mu \geq \tfrac{\beta\gamma}{(\beta-1)\gamma - \beta}$).
Moreover, $P_4^* = P_3^*$ if and only if $\mu^{-1}+\beta^{-1}+\gamma^{-1} = 1.$
\end{enumerate}
\end{lemma}

\begin{proof}
The statements concerning $P_1^*$ and $P_2^*$ follow straightforwardly
from their formulae (see also Figure~\ref{fig:EntradesISortides})
since, 
for $\mu \geq 1,$ $\frac{\mu - 1}{\mu} \in \left[0, \tfrac{3}{4}\right].$

For the fixed point $P_3^*$ when
$\tfrac{1}{\mu} + \tfrac{1}{\beta} \leq 1$
we have
$0 < \frac{1}{\beta},$
$0 \leq 1 - \frac{1}{\mu} - \frac{1}{\beta},$ and
$
 \frac{1}{\beta} + \left(1 - \frac{1}{\mu} - \frac{1}{\beta}\right) = 
 1 - \frac{1}{\mu} < 1.
$
So, 
\[
  P_3^* \in 
  \{(x,y,0)\in\R^3\, \colon x,y \geq 0 \text{ and } x+y \leq 1\} \subset
  \{(x,y,0) \in \RD\} \subset \RD.
\]
Moreover, when $\tfrac{1}{\mu} + \tfrac{1}{\beta} = 1$ we have
\[
  P_3^* = 
    \left(\frac{1}{\beta}, 1 - \frac{1}{\mu} - \frac{1}{\beta}, 0 \right) =
    \left(1-\frac{1}{\mu}, 0, 0 \right) = P_2^*.
\]

For the fixed point 
$
P_4^* = \left(
      \tfrac{1}{2} \left( 1 - \mu^{-1} + \beta^{-1} - \gamma^{-1} \right),
      \tfrac{1}{\gamma},
      \tfrac{1}{2} \left( 1 - \mu^{-1} - \beta^{-1} - \gamma^{-1} \right)
      \right)
$ 
when
$\mu^{-1}+\beta^{-1}+\gamma^{-1} \leq 1$
we have
\begin{align*}
0 & < \frac{1}{\gamma},\\
0 & \leq \frac{1}{2} \left( 1 - \mu^{-1} - \beta^{-1} - \gamma^{-1} \right),\\
0 & < \beta^{-1} \leq \frac{1}{2} \left( 1 - \mu^{-1} + \beta^{-1} - \gamma^{-1} \right),\text{ and}\\
  & \phantom{<} 
    \frac{1}{2} \left( 1 - \mu^{-1} + \beta^{-1} - \gamma^{-1} \right) +
    \frac{1}{\gamma} +
    \frac{1}{2} \left( 1 - \mu^{-1} - \beta^{-1} - \gamma^{-1} \right) =
    1 - \frac{1}{\mu} < 1.
\end{align*}
So, 
\[
  P_4^* \in
  \{(x,y,z)\in\U \, \colon y > 0 \text{ and } x \geq z\} \subset\RD.
\]
Moreover, when $\mu^{-1}+\beta^{-1}+\gamma^{-1} = 1$
($1 - \mu^{-1} - \beta^{-1} - \gamma^{-1} = 0$) we have
\begin{align*}
&\gamma^{-1} = 1 - \mu^{-1}-\beta^{-1},\text{ and}\\
&1 - \mu^{-1} + \beta^{-1} - \gamma^{-1} = 
 1 - \mu^{-1} - \beta^{-1} - \gamma^{-1} + 2\beta^{-1} =
 2\beta^{-1}.
\end{align*} 
So,
\begin{multline*}
P_4^* = \left(
      \frac{1}{2} \left( 1 - \mu^{-1} + \beta^{-1} - \gamma^{-1} \right),
      \frac{1}{\gamma},
      \frac{1}{2} \left( 1 - \mu^{-1} - \beta^{-1} - \gamma^{-1} \right)
      \right) = \\ 
      \left(\beta^{-1}, 1 - \mu^{-1}-\beta^{-1}, 0 \right) = P_3^*.
\end{multline*}
\end{proof}

Henceforth, this section will be devoted to the study of the local
stability and dynamics around the fixed points $P_1^*,\ldots,P_4^*$
for parameters moving in $\Q.$
This work is carried out by means of four lemmas (Lemmas~\ref{lem:P1}
to~\ref{lem:P4}). The information provided by them is summarized
graphically in Figure~\ref{SuPeRfIgUrE} and Figure~\ref{diag}.

The study of the stability around the fixed points is based on the
computation of the eigenvalues of its Jacobian matrix. In our case,
the Jacobian matrix of map~\eqref{eq:sistema} at a point $(x, y, z)$
is
\[
J(x,y,z) = \begin{pmatrix}
   \mu(1-2x-y-z) & -\mu x      & -\mu x \\
   \beta y       & \beta (x-z) & -\beta y \\
   0             & \gamma z    & \gamma y
\end{pmatrix}
\]
and has determinant 
$\det(J(x,y,z)) = \mu \beta \gamma xy ( 1 - 2x-2z).$
and has determinant $\det(J(x,y,z)) = \mu \beta \gamma xy ( 1 - 2x-2z).$

The first one of these lemmas follows from a really simple computation.

\begin{lemma}\label{lem:P1}
The point $P_1^* = (0,0,0)$ is a boundary fixed point of
system~\eqref{eq:sistema} for any positive $\mu$, $\beta$, $\gamma.$
Moreover, $P_1^*$ is:
\begin{itemize}
\item non-hyperbolic when $\mu = 1,$
\item a locally asymptotically stable sink node when $0 < \mu < 1$, and
\item a saddle with an unstable manifold of dimension 1, locally tangent
to the $x$-axis, when $\mu > 1.$
\end{itemize}
\end{lemma}

\begin{proof}
The Jacobian matrix of system~\eqref{eq:sistema} at $P_1^*$ is
\[
J(P_1^*) = \begin{pmatrix}
  \mu & 0 & 0 \\
    0 & 0 & 0 \\
    0 & 0 & 0
\end{pmatrix},
\]
which has an eigenvalue
$\lambda_{1,1} = \mu$ with eigenvector $(1,0,0),$
and two eigenvalues
$\lambda_{1,2} = \lambda_{1,3} = 0$ with eigenvectors $(0,1,0)$ and $(0,0,1)$
(see Figure~\ref{SuPeRfIgUrE}). 
Hereafter we will label the eigenvalues $j$ of a fixed point 
$P_i^*$ as $\lambda_{i,j}$, with $i = 1,\dots,4$ and $j = 1,\dots,3.$
The assertion of the lemma follows from the Hartman-Grobman Theorem.
\end{proof}

\begin{lemma}\label{lem:P2}
The point $P_2^* = \left(1-\mu^{-1}, 0, 0\right)$ is a boundary
fixed point of the system~\eqref{eq:sistema} for all parameters such
that $\mu^{-1}\leq1.$ 
In particular, for all of the values of the parameters in $\Q$,
the fixed point $P_2^*$ is \emph{non-hyperbolic} if and only if:
\begin{itemize}
\item $\mu = 1,$ that is, when $P_2^* = P_1^*$;
\item $\mu = \tfrac{\beta}{\beta-1}$, that is, when $P_2^* = P_3^*$;
\item $\mu = 3.$
\end{itemize}
The region of the parameter's cuboid where $P_2^*$ is \emph{hyperbolic} is
divided into the following three layers:
\begin{labeledlist}{$\tfrac{\beta}{\beta-1} < \mu < 3$}
\item[$1 < \mu < \tfrac{\beta}{\beta-1}$]
   $P_2^*$ is a locally asymptotically stable sink node, meaning that
   the two predator species go to extinction.
\item[$\tfrac{\beta}{\beta-1} < \mu < 3$]
   $P_2^*$ is a saddle with an unstable manifold of dimension 1
   locally tangent to the $x$-axis.
\item[$3 < \mu \le 4$]
   $P_2^*$ is a saddle with an unstable manifold of dimension 2
   locally tangent to the plane generated by the vectors $(1,0,0)$
   and $\left(1, \tfrac{2-\mu}{\mu-1} - \tfrac{\beta}{\mu}, 0\right).$
\end{labeledlist}\smallskip
From Lemma \ref{lem:P1} it is clear that the fixed points 
$P_{1,2}^*$ and $P_{2,3}^*$
undergo a transcritical bifurcation at 
$\mu = 1$ and $\mu = \tfrac{\beta}{\beta - 1}$
(in other words, $\mu^{-1}+\beta^{-1} = 1$),
respectively.
\end{lemma}

\begin{proof}
The Jacobian matrix of system~\eqref{eq:sistema} at $P_2^*$ is
\[
J(P_2^*) = \begin{pmatrix}
   2 - \mu & 1 - \mu                            & 1 - \mu \\
         0 & \beta\left(1-\tfrac{1}{\mu}\right) & 0 \\
         0 &                                  0 & 0
\end{pmatrix},
\]
which has:
\begin{itemize}
\item an eigenvalue $\lambda_{2,1} = 2 - \mu$ with eigenvector $(1,0,0),$
\item an eigenvalue $\lambda_{2,2} = \beta\left(1-\tfrac{1}{\mu}\right)$
with eigenvector 
$\left(1, \tfrac{2-\mu}{\mu-1} - \tfrac{\beta}{\mu}, 0\right),$
\item and an eigenvalue $\lambda_{2,3} = 0$ with eigenvector
$\left(1, 0, \tfrac{2-\mu}{\mu-1}\right).$
\end{itemize}
Moreover, for $1 \le \mu \le 4$ we have
$2 - \mu \in [-2, 1]$ and
$\beta\left(1-\tfrac{1}{\mu}\right) \ge 0.$
Clearly (see Figure~\ref{SuPeRfIgUrE}) one has:
\begin{itemize}
\item $2 - \mu = \pm 1$ if and only if $\mu = 2 \mp 1,$ and 
      $\abs{2 - \mu} < 1$ if and only if $1 < \mu < 3;$
\item $\beta\left(1-\frac{1}{\mu}\right) = 1$ if and only if 
      $\mu = \tfrac{\beta}{\beta - 1},$ and 
      $0 \le \beta\left(1-\frac{1}{\mu}\right) < 1$ if and only if 
      $1 \le \mu < \tfrac{\beta}{\beta-1};$
\item $1 < \tfrac{\beta}{\beta-1} < 3$ for every $\beta \in [2.5,5].$
\end{itemize}
Then the lemma follows from the Hartman-Grobman Theorem.
\end{proof}

\begin{lemma}\label{lem:P3}
The point
$P_3^* = \left(\beta^{-1}, 1-\beta^{-1}-\mu^{-1}, 0\right)$
is a boundary fixed point of the system~\eqref{eq:sistema} for all
positive parameters such that
$\mu^{-1} + \beta^{-1} \leq 1.$ 
In particular, for all the parameters in $\Q$, the fixed point 
$P_3^*$ is \emph{non-hyperbolic} if and only if:
\begin{itemize}
\item $\mu = \tfrac{\beta}{\beta-1},$ that is, when $P_3^* = P_2^*$;
\item $\mu = \tfrac{\beta\gamma}{(\beta-1)\gamma - \beta}$, that is,
      when $P_3^* = P_4^*$;
\item $\mu = \frac{\beta}{\beta-2}.$
\end{itemize}
The region in $\Q$ where $P_3^*$ is \emph{hyperbolic} is divided into
the following four layers:
\begin{labeledlist}{$2\beta\left(\beta-1 - \sqrt{\beta(\beta-2)}\right) < \mu < \tfrac{\beta\gamma}{(\beta - 1)\gamma - \beta}$}
\item[$\tfrac{\beta}{\beta-1} < \mu \le 2\beta\left(\beta-1 - \sqrt{\beta(\beta-2)}\right)$]
   $P_3^*$ is a locally asymptotically stable sink node. Here preys $x$
   and predators $y$ achieve a static equilibrium.
\item[$2\beta\left(\beta-1 - \sqrt{\beta(\beta-2)}\right) < \mu < \tfrac{\beta\gamma}{(\beta - 1)\gamma - \beta}$]
   $P_3^*$ is a locally asymptotically stable spiral-node sink. 
   Here preys $x$ and predators $y$ achieve also a static equilibrium,
   reached via damped oscillations.
\item[$\tfrac{\beta\gamma}{(\beta - 1)\gamma - \beta} < \mu < \min\left\{4, \frac{\beta}{\beta-2}\right\}$]
   $P_3^*$ is an unstable spiral-sink node-source.
\item[$\min\left\{4, \frac{\beta}{\beta-2}\right\}<\mu\le 4$]
   $P_3^*$ is an unstable spiral-node source.
\end{labeledlist}\smallskip
From the previous calculations one has that the fixed points
$P_{2,3}^*$ undergo a transcritical bifurcation at 
$\mu = \tfrac{\beta}{\beta-1}$ (see also Lemma~\ref{lem:P2}) and 
$\mu = \tfrac{\beta\gamma}{(\beta - 1)\gamma - \beta}$,
respectively\footnote{In terms of inverses of the parameters, these identities read 
$\mu^{-1}+\beta^{-1} = 1$ and 
$\mu^{-1}+\beta^{-1}+\gamma^{-1} = 1.$}.
\end{lemma}

\begin{proof}[Proof of Lemma~\ref{lem:P3}]
The Jacobian matrix of system~\eqref{eq:sistema} at $P_3^*$ is
\[
J(P_3^*) = \begin{pmatrix}
     1 - \tfrac{\mu}{\beta}                  & -\tfrac{\mu}{\beta} & -\tfrac{\mu}{\beta} \\
     \beta\left(1 - \tfrac{1}{\mu}\right) -1 &                   1 & \beta\left(\tfrac{1}{\mu} - 1\right)+1 \\
     0                                       &                   0 & \gamma\left(1 - \tfrac{1}{\beta} - \tfrac{1}{\mu}\right)
\end{pmatrix},
\]
and has eigenvalues
\begin{align*}
 \lambda_{3,1} &= \gamma\left(1 - \tfrac{1}{\beta} - \tfrac{1}{\mu}\right),\text{ and}\\
 \lambda_{3,2},\lambda_{3,3} &= 1 - \frac{\mu}{2\beta} \pm \frac{\sqrt{\left(\beta + \tfrac{\mu}{2}\right)^2 - \beta^2\mu}}{\beta}.
\end{align*}
For $\tfrac{\beta}{\beta-1} \le \mu \le 4$ we have
  $\gamma\left(1 - \tfrac{1}{\beta} - \tfrac{1}{\mu}\right) \ge 0$ and
  $\gamma\left(1 - \tfrac{1}{\beta} - \tfrac{1}{\mu}\right) = 1$ if and only if
  $\mu = \tfrac{\beta\gamma}{(\beta - 1)\gamma - \beta}$
(see Figure~\ref{SuPeRfIgUrE}).
On the other hand,
$\left(\beta + \tfrac{\mu}{2}\right)^2 - \beta^2\mu = 0$
if and only if
$\mu = 2\beta\left(\beta - 1 \pm \sqrt{\beta(\beta-2)}\right).$

Now let us study the relation between
  $\tfrac{\beta}{\beta-1},$
  $2\beta\left(\beta-1 \pm \sqrt{\beta(\beta-2)}\right),$
  $\tfrac{\beta\gamma}{(\beta - 1)\gamma - \beta}$ and
  $\frac{\beta}{\beta-2}.$
First we observe that, since $\beta \ge 2.5,$
\[
   2\beta\left(\beta - 1 + \sqrt{\beta(\beta-2)}\right) \ge 5(1.5 + \sqrt{1.25}) > 13 > 4 \ge \mu.
\]
Consequently, we simultaneously have
$\left(\beta + \tfrac{\mu}{2}\right)^2 - \beta^2\mu = 0$ and $\mu \le 4$
if and only if
$\mu = 2\beta\left(\beta - 1 - \sqrt{\beta(\beta-2)}\right).$

Second, since $\beta(\beta-2) = (\beta-1)^2 - 1,$ it follows that
\[
 \beta(\beta-2) < (\beta-1)^2 - 1 + \frac{1}{4(\beta-1)^2} =
                  \left((\beta-1) - \tfrac{1}{2(\beta-1)}\right)^2.
\]
Moreover,
  $\beta(\beta-2) > 0$ and
  $(\beta-1) - \tfrac{1}{2(\beta-1)} > 0$
     (which follows from the inequality $2(\beta-1)^2  > 1$).
So, the above inequality is equivalent to
\[
 \sqrt{\beta(\beta-2)} < (\beta-1) - \frac{1}{2(\beta-1)}
 \qquad\Longleftrightarrow\qquad
 \frac{1}{2(\beta-1)} < (\beta-1) - \sqrt{\beta(\beta-2)}
\]
which, in turn, is equivalent to
\[
 \frac{\beta}{\beta-1} < 2\beta\left(\beta-1 - \sqrt{\beta(\beta-2)}\right).
\]
Third, we will show that
\begin{equation}\label{eq:UFFFFF}
 2\beta\left(\beta-1 - \sqrt{\beta(\beta-2)}\right) < \frac{\beta\gamma}{(\beta - 1)\gamma - \beta}.
\end{equation}
To this end observe that
\begin{equation}\label{eq:TheDerivative}
 \frac{\partial}{\partial\gamma} \frac{\gamma}{(\beta - 1)\gamma - \beta} =
 -\frac{\beta}{((\beta - 1)\gamma - \beta)^2} < 0.
\end{equation}
Hence, by replacing $9.4$ by $\tfrac{47}{5},$
\[
 \frac{47}{42\beta - 47} = \frac{\tfrac{47}{5}}{\tfrac{47(\beta - 1)}{5} - \beta} \le \frac{\gamma}{(\beta - 1)\gamma - \beta}.
\]
So, to prove \eqref{eq:UFFFFF}, it is enough to show that
\[
    \beta-1 - \sqrt{\beta(\beta-2)} <
    \frac{47}{2(42\beta - 47)} \le
    \frac{1}{2\beta}\frac{\beta\gamma}{(\beta - 1)\gamma - \beta}.
\]
This inequality is equivalent to
\[
 \frac{84\beta^2 - 178\beta + 47}{84\beta - 94} =
 \beta-1 - \frac{47}{84\beta - 94} <
 \sqrt{\beta(\beta-2)}.
\]
Since $\beta \ge 2.5,$
 $\tfrac{84\beta^2 - 178\beta + 47}{84\beta - 94}$ is positive
and thus, it is enough to prove that
\[
 \frac{(84\beta^2 - 178\beta + 47)^2}{(84\beta - 94)^2} < \beta(\beta-2),
\]
which is equivalent to
\[
 0 < \beta(\beta-2)(84\beta - 94)^2 - (84\beta^2 - 178\beta + 47)^2 =
     840\beta^2-940\beta-2209.
\]
This last polynomial is positive for every
\[ \beta > \frac{47\sqrt{235} + 235}{420} \approx 2.27499\cdots\;. \]
This ends the proof of \eqref{eq:UFFFFF}.
On the other hand, since $\gamma \ge 5 \ge \beta,$ we have
$-\beta\gamma \le -\beta^2$ which is equivalent to
\[
  \beta\gamma(\beta-2) = \beta^2\gamma - 2\beta\gamma \le
  \beta^2\gamma - \beta\gamma - \beta^2 = \beta((\beta - 1)\gamma - \beta),
\]
and this last inequality is equivalent to
\[
  \frac{\beta\gamma}{(\beta - 1)\gamma - \beta} \le \frac{\beta}{\beta-2}
\]
(with equality only when $\gamma = 5 = \beta$).
Thus, summarizing, we have seen:
\begin{equation}\label{eq:ordering}
 1 < \frac{\beta}{\beta-1} <
 2\beta\left(\beta-1 - \sqrt{\beta(\beta-2)}\right) <
 \frac{\beta\gamma}{(\beta - 1)\gamma - \beta} \le \frac{\beta}{\beta-2}
\end{equation}
and the last inequality is an equality only when $\gamma = 5 = \beta.$

Next we study the modulus of the eigenvalues to determine the local
stability of $P_3^*.$
First observe (see Figure~\ref{SuPeRfIgUrE}) that
$\abs{\lambda_{3,1}} = \abs{\gamma\left(1 - \tfrac{1}{\beta} - \tfrac{1}{\mu}\right)} < 1$
if and only if
$\tfrac{\beta}{\beta-1} < \mu < \tfrac{\beta\gamma}{(\beta - 1)\gamma - \beta}.$
On the other hand, on
$\tfrac{\beta}{\beta-1} < \mu \le 2\beta\left(\beta-1 - \sqrt{\beta(\beta-2)}\right),$
the discriminant
$\left(\beta + \tfrac{\mu}{2}\right)^2 - \beta^2\mu$ is non-negative
and the eigenvalues $\lambda_{3,2}$ and $\lambda_{3,3}$ are real.
Moreover, $\beta > 2$ is equivalent to
$-\beta\mu > \beta\mu - \beta^2\mu$ and this to
\[
 \left(\beta - \tfrac{\mu}{2}\right)^2 > \left(\beta + \tfrac{\mu}{2}\right)^2 - \beta^2\mu
 \quad\Leftrightarrow\quad
 \beta - \frac{\mu}{2} > \sqrt{\left(\beta + \tfrac{\mu}{2}\right)^2 - \beta^2\mu}
\]
(observe that $\beta - \tfrac{\mu}{2} > 0$
because $\beta > 2$ and $\mu \le 4,$ and recall that
$\left(\beta + \tfrac{\mu}{2}\right)^2 - \beta^2\mu$
is non-negative in the selected region).
The last inequality above is equivalent to
\[
 1 - \frac{\mu}{2\beta} > \frac{\sqrt{\left(\beta + \tfrac{\mu}{2}\right)^2 - \beta^2\mu}}{\beta}
 \quad\Leftrightarrow\quad
 0 < 1 - \frac{\mu}{2\beta} - \frac{\sqrt{\left(\beta + \tfrac{\mu}{2}\right)^2 - \beta^2\mu}}{\beta} = \lambda_{3,3}.
\]
On the other hand, the following equivalent expressions hold:
\begin{eqnarray*}
&&\frac{\beta}{\beta-1} < \mu \Leftrightarrow \beta^2 < \mu\beta(\beta-1) \Leftrightarrow \left(\beta + \tfrac{\mu}{2}\right)^2 - \beta^2\mu =
 \beta^2 + \frac{\mu^2}{4} + \mu\beta - \mu\beta^2 <
 \frac{\mu^2}{4} \\
&&\frac{\sqrt{\left(\beta + \tfrac{\mu}{2}\right)^2 - \beta^2\mu}}{\beta} < \frac{\mu}{2\beta}
 \Leftrightarrow
 \lambda_{3,2} = 1 - \frac{\mu}{2\beta} + \frac{\sqrt{\left(\beta + \tfrac{\mu}{2}\right)^2 - \beta^2\mu}}{\beta} < 1.   
\end{eqnarray*}
Summarizing, when
$\tfrac{\beta}{\beta-1} < \mu \le 2\beta\left(\beta-1 - \sqrt{\beta(\beta-2)}\right)$
we have
\[
  0 < \lambda_{3,3} =
  1 - \frac{\mu}{2\beta} - \frac{\sqrt{\left(\beta + \tfrac{\mu}{2}\right)^2 - \beta^2\mu}}{\beta} \le
  1 - \frac{\mu}{2\beta} + \frac{\sqrt{\left(\beta + \tfrac{\mu}{2}\right)^2 - \beta^2\mu}}{\beta} =
  \lambda_{3,2} < 1.
\]
Consequently, $P_3^*$ is a locally asymptotically stable sink node by \eqref{eq:ordering}, 
meaning that top predators ($z$) go to extinction and the other two species persist.

Now we consider the region
$2\beta\left(\beta-1 - \sqrt{\beta(\beta-2)}\right) < \mu \le 4.$
In this case the discriminant
$\left(\beta + \tfrac{\mu}{2}\right)^2 - \beta^2\mu$ is negative
and the eigenvalues $\lambda_{3,2}$ and $\lambda_{3,3}$ are complex
conjugate with modulus
\[
 \sqrt{\left(1 - \tfrac{\mu}{2\beta}\right)^2 + \frac{\beta^2\mu - \left(\beta + \tfrac{\mu}{2}\right)^2}{\beta^2}} =
 \sqrt{\frac{\mu(\beta-2)}{\beta}}.
\]
Clearly
\[
 \abs{\sqrt{\frac{\mu(\beta-2)}{\beta}}} < 1
 \quad\Leftrightarrow\quad
2\beta\left(\beta-1 - \sqrt{\beta(\beta-2)}\right) \le \mu \le \min\left\{4, \frac{\beta}{\beta-2}\right\}
\]
(with equality only when $\gamma = 5 = \beta$).
Then the lemma follows from the Hartman-Grobman Theorem.
\end{proof}

\begin{lemma}\label{lem:P4}
The point
$P_4^* = \left(\rho, \gamma^{-1}, \rho - \beta^{-1}\right)$
with 
$\rho = \tfrac{1}{2}\left(1+\beta^{-1}- \gamma^{-1}-\mu^{-1}\right)$
is a fixed point of the system~\eqref{eq:sistema} for all
positive parameters satisfying that 
$\mu^{-1} + \beta^{-1} + \gamma^{-1} \leq 1.$
Moreover, for all the parameters in $\Q$, there exists a function
$
 \psi_4 \colon [2.5,5] \times [5, 9.4] \longrightarrow
 \left[\tfrac{\beta\gamma}{(\beta - 1)\gamma - \beta}, 4\right)
$
(whose graph $\Sigma_4$ is drawn in 
\textcolor{red!80!green}{redish}  colour in Figure~\ref{SuPeRfIgUrE})
such that $P_4^*$ is \emph{non-hyperbolic} if and only if:
\begin{itemize}
\item $\mu = \tfrac{\beta\gamma}{(\beta - 1)\gamma - \beta}$, that is, 
      when $P_4^* = P_3^*$;
\item $\mu = \psi_4(\beta, \gamma).$ 
\end{itemize}
Furthermore, the region of the parameter's cuboid where $P_4^*$ is 
\emph{hyperbolic} is divided into the following two layers:
\begin{labeledlist}{$\tfrac{\beta\gamma}{(\beta - 1)\gamma - \beta} < \mu < \psi_4(\beta,\gamma)$}
\item[$\tfrac{\beta\gamma}{(\beta - 1)\gamma - \beta} < \mu < \psi_4(\beta,\gamma)$]
     $P_4^*$ is a locally asymptotically stable sink of spiral-node type. 
     Within this first layer the three species achieve a static coexistence 
     equilibrium with an oscillatory transient.
\item[$\psi_4(\beta,\gamma) < \mu \le 4$]
     $P_4^*$ is an unstable spiral-source node-sink.
\end{labeledlist}
\end{lemma}

\begin{proof}
The Jacobian matrix of system~\eqref{eq:sistema} at $P_4^*$ is
\[
 J(P_4^*) = \begin{pmatrix}
    1-\mu\rho             & -\mu\rho                                 & -\mu\rho \\
    \tfrac{\beta}{\gamma} &        1                                 & -\tfrac{\beta}{\gamma}\\
                        0 & \gamma\left(\rho-\tfrac{1}{\beta}\right) & 1
\end{pmatrix}\; .
\]
The matrix $J(P_4^*)$ has eigenvalues
\begin{align*}
 \lambda_{4,1} &:= 1 - \frac{\mu\rho}{3} + 
                   \frac{\sqrt[3]{\alpha}}{3\sqrt[3]{2}\sqrt{\gamma}} + 
                   \frac{\sqrt[3]{2}\left(\mu^2\rho^2\gamma - 
                   3\beta(\mu+\gamma)\rho + 
                   3\gamma\right)}{3\sqrt{\gamma}\sqrt[3]{\alpha}},\\
 \lambda_{4,2}, \lambda_{4,3} &:= 1 - \frac{\mu\rho}{3} - 
                                  \frac{\sqrt[3]{\alpha}}{3\sqrt[3]{16}\sqrt{\gamma}}(1 \mp \sqrt{3}i) - 
                                  \frac{\mu^2\rho^2\gamma - 
                                  3\beta(\mu+\gamma)\rho + 
                                  3\gamma}{3\sqrt[3]{4}\sqrt{\gamma}\sqrt[3]{\alpha}}\left(1 \pm \sqrt{3}i\right),\\
\end{align*}
where
\begin{align*}
\alpha =&\ -2  \gamma^{3/2} \rho^3 \mu^3 -
            45 \gamma^{3/2} \beta \mu \rho^2 +
            9  \mu^2 \rho^2 \beta \sqrt{\gamma} +
            45 \gamma^{3/2} \rho \mu + 
            \sqrt{27\widetilde{\alpha}},\text{ and}\\
\widetilde{\alpha} :=&\ 8( \rho\,\beta-1)
                        \left(\mu^4 \rho^4 + 
                        \frac{71(\rho\beta - 1) \rho^2 \mu^2}{8} + 
                        \frac{(\beta\rho-1)^2}{2}\right)\gamma^3\\
                    &\ -38(\rho\beta-1)\left(\mu^2 \rho^2 -
                        \frac{6(\rho\beta-1)}{19}\right) \rho \beta \mu \gamma^2\\
                    &\ -\mu^2 \rho^2 \beta^2 \left(\mu^2 \rho^2 -12(\rho\beta-1)\right) \gamma +
                        4 \beta^3 \mu^3 \rho^3.
\end{align*}
When $\mu = \tfrac{\beta\gamma}{(\beta - 1)\gamma - \beta}$ we clearly have
\[
 J(P_4^*) = \begin{pmatrix}
    1-\mu\rho             & -\mu\rho & -\mu\rho \\
    \tfrac{\beta}{\gamma} &        1 & -\tfrac{\beta}{\gamma}\\
                        0 &        0 & 1
\end{pmatrix}
\]
and
\[
 \begin{pmatrix}
     1-\mu\rho             & -\mu\rho & -\mu\rho \\
     \tfrac{\beta}{\gamma} &        1 & -\tfrac{\beta}{\gamma}\\
                         0 &        0 & 1
\end{pmatrix} 
\begin{pmatrix} \phantom{+}1 \\ -2 \\ \phantom{+}1 \end{pmatrix} =
\begin{pmatrix} \phantom{+}1 \\ -2 \\ \phantom{+}1 \end{pmatrix}\;.
\]
Thus, for every $(\beta, \gamma) \in [2.5, 5] \times [5, 9.4],$
$\lambda_{4,1} = 1$ when
$\mu = \tfrac{\beta\gamma}{(\beta - 1)\gamma - \beta}.$
Moreover, it can be seen numerically that
for every $(\beta, \gamma) \in [2.5, 5] \times [5, 9.4],$
$\lambda_{4,1}$ is a strictly decreasing function of $\mu$ such that
$\lambda_{4,1} > -1$ when $\mu = 4.$
So, $\lambda_{4,1}$ only breaks the hyperbolicity of $P^*_4$ in
the surface $\mu = \tfrac{\beta\gamma}{(\beta - 1)\gamma - \beta}$
and $\abs{\lambda_{4,1}} < 1$ in the region
\[
   (\beta, \gamma, \mu) \in [2.5, 5] \times [5, 9.4] \times
        \left(\tfrac{\beta\gamma}{(\beta - 1)\gamma - \beta}, 4\right].
\]

Next we need to describe the behaviour of
$\abs{\lambda_{4,2}} = \abs{\lambda_{4,3}}$ as a function of $\mu.$
The following statements have been observed numerically:
\begin{enumerate}[(i)]
\item $\abs{\lambda_{4,2}} = \abs{\lambda_{4,3}} < 1$ for every
      \[ (\beta, \gamma, \mu) \in [2.5, 5] \times [5, 9.4] \times
         \left\{\tfrac{\beta\gamma}{(\beta - 1)\gamma - \beta}\right\}
         \setminus
         \left\{\left(5, 5, \tfrac{\beta\gamma}{(\beta - 1)\gamma - \beta}\right)\right\}
      \] and $\abs{\lambda_{4,2}} = \abs{\lambda_{4,3}} = 1$
      at the point
      $(\beta, \gamma, \mu) = \left(5, 5, \tfrac{\beta\gamma}{(\beta - 1)\gamma - \beta}\right).$
\item $\abs{\lambda_{4,2}} = \abs{\lambda_{4,3}} > 1$ for every
      $(\beta, \gamma, \mu) \in [2.5, 5] \times [5, 9.4] \times \{4\}.$
\end{enumerate} \begin{enumerate}[{(iii.}1)]
\item There exists a Non-Monotonic ($\textsf{NM}$) region
$\textsf{NM}_4 \subset [2.5, 2.59597\cdots] \times [5, 6.49712\cdots]$ 
such that for every $(\beta, \gamma) \in \textsf{NM}_4$
there exists a value
$
  \mu^*(\beta,\gamma) \in \left(\tfrac{\beta\gamma}{(\beta - 1)\gamma - \beta}, 4\right)
$
with the property that $\abs{\lambda_{4,2}} = \abs{\lambda_{4,3}}$
is a strictly decreasing function of the parameter
$
   \mu \in \left[ \tfrac{\beta\gamma}{(\beta - 1)\gamma - \beta}, \mu^*(\beta,\gamma) \right],
$
and a strictly increasing function of $\mu$ for every value of 
$\mu \in \left[ \mu^*(\beta,\gamma), 4 \right].$
In particular, from (i) it follows that 
$\abs{\lambda_{4,2}} = \abs{\lambda_{4,3}} < 1$ holds for every point
$
   (\beta, \gamma, \mu) \in \textsf{NM}_4 \times \left[ 
       \tfrac{\beta\gamma}{(\beta - 1)\gamma - \beta}, \mu^*(\beta,\gamma) 
   \right].
$
\linebreak
\begin{minipage}{0.5\textwidth}
Consequently, for every $(\beta, \gamma) \in \textsf{NM}_4,$
there exists a unique value of the parameter
$\mu = \psi_4(\beta,\gamma) > \mu^*(\beta,\gamma)$ such that
$\abs{\lambda_{4,2}} = \abs{\lambda_{4,3}} = 1$ at the point
$(\beta, \gamma, \psi_4(\beta,\gamma)).$

\medskip

The region $\textsf{NM}_4,$ as shown in the picture at the side,
is delimited by the axes $\beta = 2.5,$ $\gamma = 5$ and, 
approximately, by the curve
$\gamma \approx 19.6981\beta^2 - 115.98\beta + 173.334.$
\end{minipage}\hfill\begin{minipage}{0.37\textwidth}\flushright
\begin{tikzpicture}[font={\tiny}]
\begin{axis}[title=The region $\textsf{NM}_4$,
             xlabel={$\beta$}, xmin=2.5, xmax=2.6,
             ylabel={$\gamma$}, ymin=5, ymax=6.5,
             scale=0.5]
\addplot[name path=funcio, blue, thick, domain=2.5:2.59597] {19.6981*x^2 - 115.98*x + 173.334};
\path[name path=eixx] (axis cs:2.5,5) -- (axis cs:2.59597,5);
\addplot[fill=blue!15] fill between [of=funcio and eixx];
\end{axis}
\end{tikzpicture}
\end{minipage}

\item For every
$
    (\beta, \gamma) \in \left( [2.5, 5] \times [5, 9.4]\right) \setminus \textsf{NM}_4,
$
$\abs{\lambda_{4,2}} = \abs{\lambda_{4,3}}$ is a strictly increasing
function of $\mu.$ In particular, from (i) it follows that there
exists a unique value of
$
   \mu = \psi_4(\beta,\gamma) > \tfrac{\beta\gamma}{(\beta - 1)\gamma - \beta}
$
such that
$\abs{\lambda_{4,2}} = \abs{\lambda_{4,3}} = 1$ at the point
$(\beta, \gamma, \psi_4(\beta,\gamma)).$
\end{enumerate}
Then the lemma follows from the Hartman-Grobman Theorem.
\end{proof}

\section{Local bifurcations: Three dimensional bifurcation diagram}
\label{sec:dim3:bif:diagram}
Due to the mathematical structure of the map \eqref{eq:sistema} and
to the number of parameters one can provide analytical information on
local dynamics within different regions of the chosen parameter
space. That is, to build a three-dimensional bifurcation diagram
displaying the parametric regions involved in the local dynamics of
the fixed points investigated above. These analyses also provide some
clues on the expected global dynamics, that will be addressed
numerically in Section~\ref{LEs}. To understand the local dynamical
picture, the next lemma
relates all the surfaces that play a role in defining the local
structural stability zones in the previous four lemmas.
It justifies the relative positions of these surfaces, shown in
Figure~\ref{SuPeRfIgUrE}.

We define
\[
   \mathsf{H}_4 := \left\{ (\beta,\gamma) \colon \psi_4(\beta,\gamma) \ge 3 \right\}.
\]

\begin{lemma}\label{lem:relations}
The relations between the surfaces defined in
Lemmas~\ref{lem:P1}--\ref{lem:P4} are the following:
\begin{enumerate}[(i)]
\item For every $(\beta, \gamma) \in [2.5, 5] \times [5, 9.4],$
\[
  1 < \frac{\beta}{\beta-1} <
 2\beta\left(\beta-1 - \sqrt{\beta(\beta-2)}\right) <
   \frac{\beta\gamma}{(\beta - 1)\gamma - \beta}.
\]
\end{enumerate}

$\mathsf{H}_4$ is the region contained in
$[2.5, 2.769\cdots] \times [5, 6.068\cdots]$ and delimited
by the axes $\beta = 2.5$ and $\gamma = 5,$ and the curve
\[ \gamma \approx 2.13725\beta^2 - 15.2038\beta + 30.7162. \]

\begin{enumerate}[(i)]\setcounter{enumi}{1}
\item For every $(\beta, \gamma) \in \left( \left[2.5, \tfrac{8}{3}\right] \times [5, 9.4]\right) \cap \mathsf{H}_4,$
\[
\frac{\beta\gamma}{(\beta - 1)\gamma - \beta} < 3 \le \psi_4(\beta,\gamma) < 4 \le \frac{\beta}{\beta-2}.
\]
On the other hand, for every $(\beta, \gamma) \in \left( \left[2.5, \tfrac{8}{3}\right] \times [5, 9.4]\right) \setminus \mathsf{H}_4,$
\[
\frac{\beta\gamma}{(\beta - 1)\gamma - \beta} < \psi_4(\beta,\gamma) < 3 < 4 \le \frac{\beta}{\beta-2}.
\]

\item For every $(\beta, \gamma) \in \left( \left(\tfrac{8}{3}, 3\right) \times [5, 9.4]\right) \cap \mathsf{H}_4,$
\[
\frac{\beta\gamma}{(\beta - 1)\gamma - \beta} < 3 \le \psi_4(\beta,\gamma) < \frac{\beta}{\beta-2} < 4
\]
For every $(\beta, \gamma) \in \left( \left(\tfrac{8}{3}, 3\right) \times [5, 9.4]\right) \setminus \mathsf{H}_4,$
\[
\frac{\beta\gamma}{(\beta - 1)\gamma - \beta} < \psi_4(\beta,\gamma) < 3 < \frac{\beta}{\beta-2} < 4.
\]

\item For every $(\beta, \gamma) \in \{3\} \times [5, 9.4],$
\[
\frac{\beta\gamma}{(\beta - 1)\gamma - \beta} < \psi_4(\beta,\gamma) < 3 = \frac{\beta}{\beta-2}.
\]

\item For every $(\beta, \gamma) \in (3, 5] \times [5, 9.4],$
\[
\frac{\beta\gamma}{(\beta - 1)\gamma - \beta} \le \psi_4(\beta,\gamma) \le \frac{\beta}{\beta-2} < 3.
\]
Moreover, all the above inequalities are strict except in the point
$(\beta, \gamma) = (5,5)$ where
$\frac{\beta\gamma}{(\beta - 1)\gamma - \beta} = \psi_4(\beta,\gamma) = \frac{\beta}{\beta-2}.$
\end{enumerate}
\end{lemma}

\begin{proof}
Statement (i) follows from Equation~\eqref{eq:ordering} and
the fact that $\mathsf{H}_4$ is the region delimited
by the axes $\beta = 2.5$ and $\gamma = 5,$ and the curve
\[ \gamma \approx 2.13725\beta^2 - 15.2038\beta + 30.7162 \]
can be checked numerically.

On the other hand, observe that for every $\gamma \in [5, 9.4]$ 
we have
\[
 \begin{cases}
 \frac{\beta}{\beta-2} \ge 4   & \text{for $\beta \in \left[2.5, \tfrac{8}{3}\right]$,}\\
 4 > \frac{\beta}{\beta-2} > 3 & \text{for $\beta \in \left(\tfrac{8}{3}, 3\right)$,}\\
 \frac{\beta}{\beta-2} = 3     & \text{for $\beta = \tfrac{8}{3}$,}\\
 3 > \frac{\beta}{\beta-2}     & \text{for $\beta \in (3, 5].$}
 \end{cases}
\]
Moreover, $7\beta > 15$ is equivalent to 
$12 \beta - 15 > 5\beta,$ and 
Equation~\eqref{eq:TheDerivative} implies
\[
 3 > \frac{5\beta}{4\beta - 5} =
     \beta\frac{5}{5(\beta - 1) - \beta} \ge
     \beta\frac{\gamma}{(\beta - 1)\gamma - \beta}.
\]

So, Statements~(ii--v) follow from these observations,
Equation~\eqref{eq:ordering} and by checking numerically the various
relations of $\psi_4(\beta,\gamma)$ with
$\mu = \frac{\beta\gamma}{(\beta - 1)\gamma - \beta},$ $\mu = 3,$ and
$\mu = \frac{\beta}{\beta-2}$ for the different regions considered
in Statements~(ii--v).
\end{proof}

\begin{landscape}
\begin{figure}
\centering
\includegraphics[width=22.5cm]{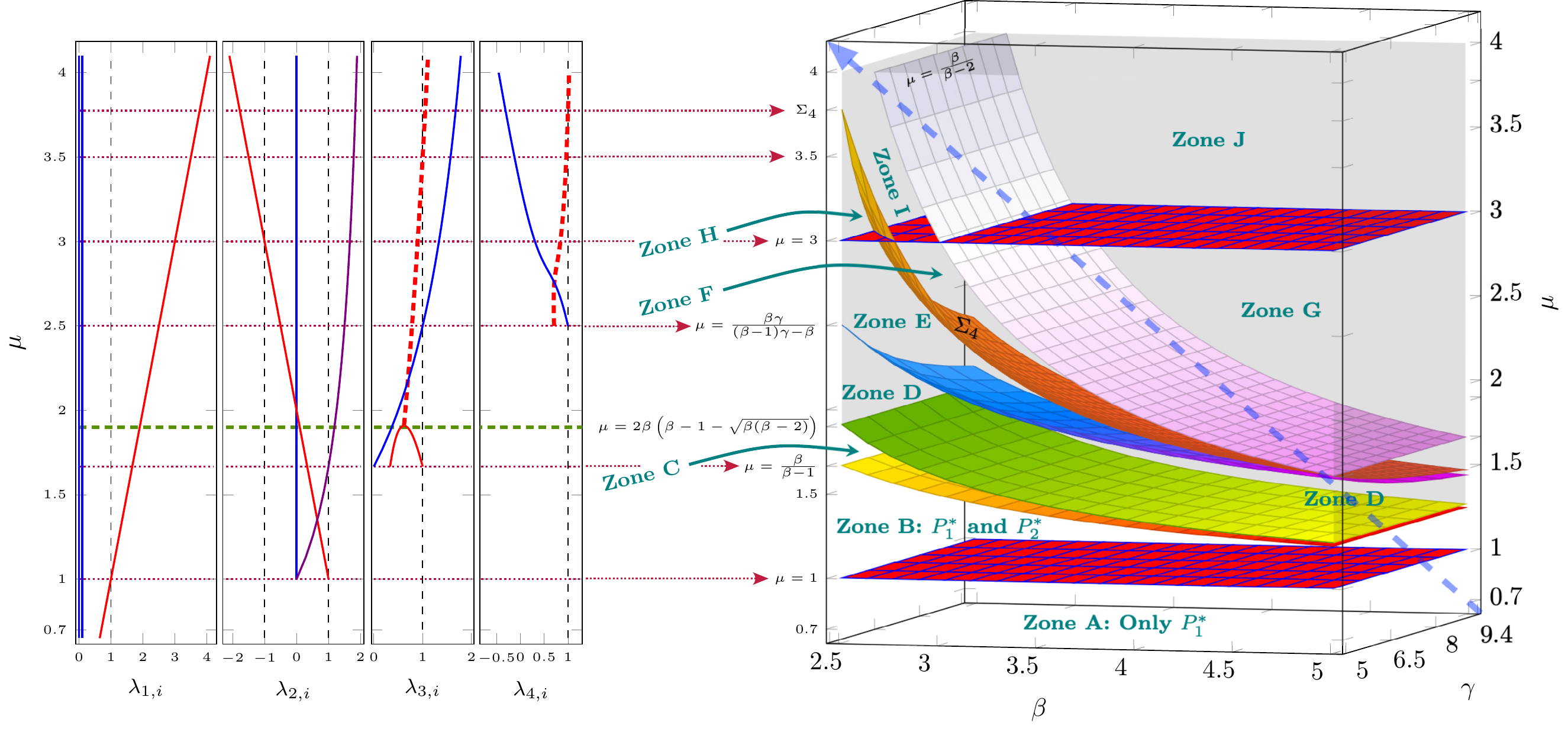}
\vspace*{-8pt}
\captionsetup{width = \linewidth}
\caption{\small 
(Left) Eigenvalues $\lambda_{j,i}$ of the fixed points 
$P^*_j$, with $j = 1,\dots,4$ and $i = 1,\dots,3$ 
(for typical values of $\beta$ and $\gamma$). 
(Right) Zones of local structural stability in the parameter space.
The 
\textcolor{blue}{blue}, 
\textcolor{red!80!green}{redish}, and
\textcolor{violet}{violet} 
surfaces intersect at the unique point $\left(5,5,\tfrac{5}{3}\right).$
For every other value of 
$(\beta, \gamma) \in [2.5,5] \times [5, 9.4]$ 
are pairwise disjoint.
The grey box above the surface 
$\mu = 2\beta\left(\beta-1-\sqrt{\beta(\beta-2)}\right)$ is the region where
the eigenvalues
\textcolor{red}{$\lambda_{3,2}$} and \textcolor{red}{$\lambda_{3,3}$}
are complex.
The eigenvalues 
\textcolor{red}{$\lambda_{4,2}$} and \textcolor{red}{$\lambda_{4,3}$} 
are complex whenever $P^*_4$ is in the positive octant.
The 
\textcolor{red}{red thick dashed lines} represent
\textcolor{red}{$\abs{\lambda_{3,2}} = \abs{\lambda_{3,3}}$} and
\textcolor{red}{$\abs{\lambda_{4,2}} = \abs{\lambda_{4,3}}$}.
In the left pictures the ``complexity region'' of
\textcolor{red}{$\lambda_{3,2}$} and \textcolor{red}{$\lambda_{3,3}$}
corresponds to the values of $\mu$ above the 
\textcolor{verdxarxadiscriminant}{green thick dashed line}.
The dynamics tied to the zones crossed by the thick, dashed, blue
arrow can be displayed in the file \textsf{Movie-3.mp4} in the
Supplementary Material.}\label{SuPeRfIgUrE}
\end{figure}
\end{landscape}

The detailed description of the local dynamics in the zones of
Figure~\ref{SuPeRfIgUrE} (see also Figure~\ref{diag}) is given by the
following (see Lemmas~\ref{lem:P1}--\ref{lem:relations}):

\begin{theorem}\label{theo:UnaCanya}
The following statements hold:
\begin{enumerate}[\bf Zone A:]
\item $(\beta, \gamma, \mu) \in [2.5, 5] \times [5, 9.4] \times (0,1).$\\
In this layer the system has $P_1^* = (0,0,0)$ as a unique
fixed point.
This fixed point is a locally asymptotically stable sink node,
meaning that the three species go to extinction. 
Indeed, it is proved in Theorem~\ref{GAS} that this is a 
globally asymptotically stable (GAS) point.
\item $(\beta, \gamma, \mu) \in [2.5, 5] \times [5, 9.4] \times \left(1, \tfrac{\beta}{\beta-1}\right).$\\
In this zone the system has exactly two fixed points: the origin
$P_1^*$ and 
$P_2^* = \left(1-\tfrac{1}{\mu}, 0, 0\right).$
$P_1^*$ is a saddle with $\mathrm{dim}\,W^u(P_1^*) = 1$ 
locally tangent to the $x$-axis and
$P_2^*$ is a locally asymptotically stable sink node. 
Hence, in this zone only preys will survive.
Theorem~\ref{GAS2} proves that in this zone $P_2^*$ is a GAS point.
\begin{figure}
\centering
\includegraphics[width=\textwidth]{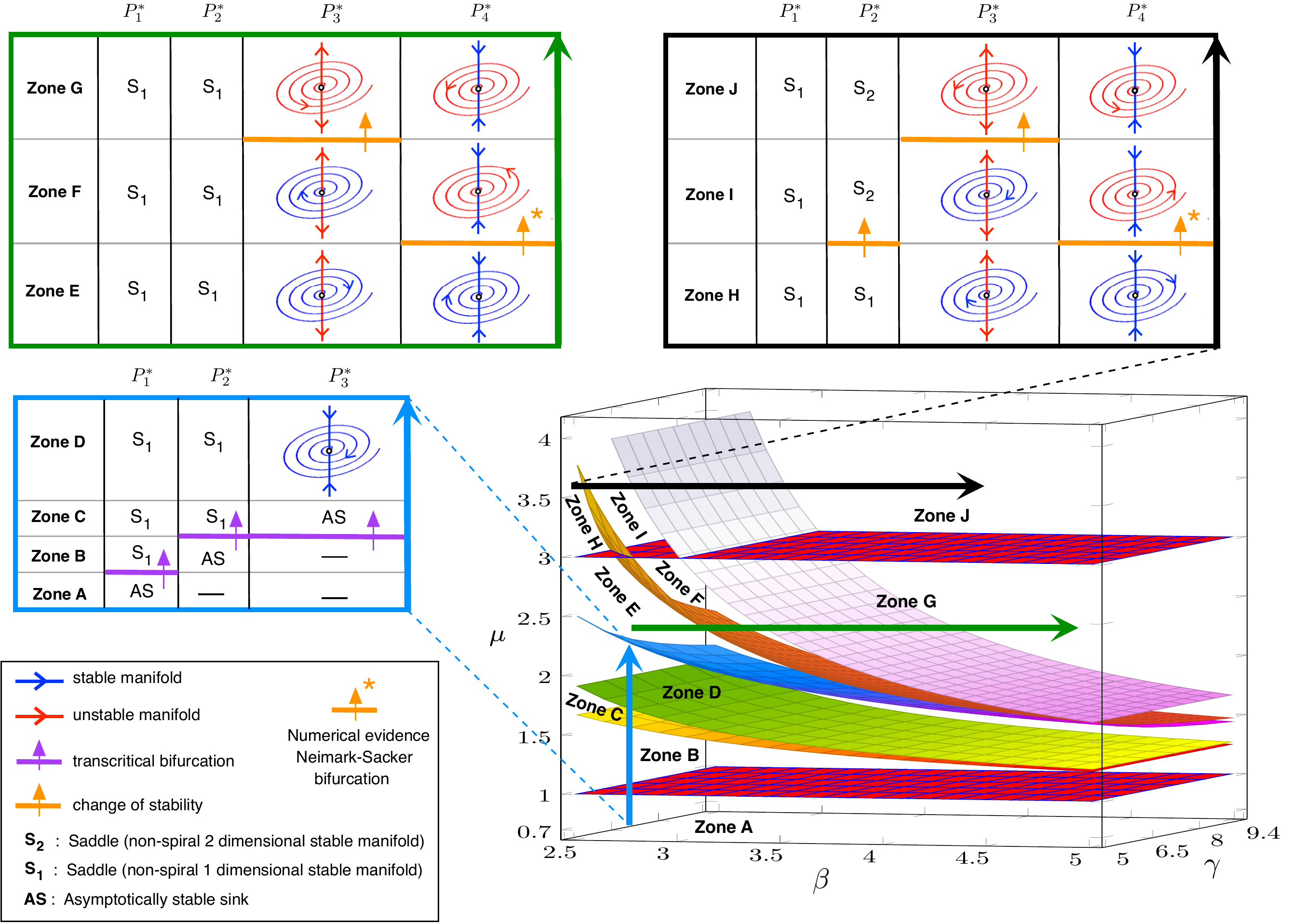}
\captionsetup{width=\linewidth}
\caption{Changes in the existence and local stability of the fixed
points tied to the transitions between the zones identified in
Figure~\ref{SuPeRfIgUrE}. The tables display, for each fixed point, the
stability nature along the thick arrows displayed in the cuboid $\Q.$
The fixed points are classified as follows: asymptotically stable
sink (AS); non-spiral saddle with a 1-dimensional $(S_1)$ and
2-dimensional $(S_2)$ stable manifold; and spirals (stable in blue;
unstable in red), see the legend below the table framed in light
blue. Stable and unstable manifolds are displayed with blue and red
arrows, respectively. The small violet arrows in the lower table
denote transcritical bifurcations, with collision of fixed points and
stability changes. The small orange arrows indicate changes in
stability without collision of fixed points. Here numerical evidences
for supercritical Neimark-Sacker bifurcations have been obtained
(indicated with an asterisk).}\label{diag}
\end{figure}
\item
$
  (\beta, \gamma, \mu) \in [2.5, 5] \times [5, 9.4] \times 
      \left(\tfrac{\beta}{\beta-1}, 2\beta\left(\beta-1 - \sqrt{\beta(\beta-2)}\right)\right).
$\\
In this region the system has exactly three fixed points:
the origin $P_1^*,$ $P_2^*$ and
$P_3^* = \left(\beta^{-1}, 1-\beta^{-1}-\mu^{-1}, 0\right).$
$P_1^*$ and $P_2^*$ are saddles with 
$\mathrm{dim}\,W^u(P_1^*,P_2^*) = 1$ and 
$P_3^*$ is a locally asymptotically stable sink node. 
Here top predators can not survive, being the system only composed of 
preys and the predator species $y.$
\item 
$
  (\beta, \gamma, \mu) \in [2.5, 5] \times [5, 9.4] \times
        \left(2\beta\left(\beta-1 - \sqrt{\beta(\beta-2)}\right),
               \tfrac{\beta\gamma}{(\beta - 1)\gamma - \beta}\right).
$\\
In this zone the system still has three fixed points:
$P_1^*,$ $P_2^*$ and $P_3^*.$
$P_1^*$ and $P_2^*$ are saddles with 
$\mathrm{dim}\,W^u(P_1^*,P_2^*) = 1$ 
but $P_3^*$ is a locally asymptotically stable spiral-node sink. 
In this region the prey and predator $y$ reach a static equilibrium of
coexistence achieved via damped oscillations, while the top predator
$z$ goes to extinction.
\item
$
  (\beta, \gamma, \mu) \in [2.5, 5] \times [5, 9.4] \times
         \left(\tfrac{\beta\gamma}{(\beta - 1)\gamma - \beta}, 
             \min\left\{3,\psi_4(\beta,\gamma)\right\}\right).
$\\
In this layer the system has exactly four fixed points:
the origin $P_1^*,$ $P_2^*,$ $P_3^*$ and
$P_4^* = \left(\rho, \gamma^{-1}, \rho - \beta^{-1}\right).$
$P_1^*$ and $P_2^*$ are saddle points with
$\mathrm{dim}\,W^u(P_1^*,P_2^*) = 1,$ 
the fixed point 
$P_3^*$ is an unstable spiral-sink node-source and
$P_4^*$ is a locally asymptotically stable sink of spiral-node type.
Under this scenario, the three species achieve a static coexistence
state also via damped oscillations.
\item 
$
  (\beta, \gamma, \mu) \in \left(\left([2.5, 5] \times 
        [5,9.4]\right)\setminus \mathsf{H}_4\right) \times
             \left(\psi_4(\beta,\gamma), 
                  \min\left\{3,\frac{\beta}{\beta-2}\right\}\right).
$\\
In this layer the system has exactly four fixed points:
the origin $P_1^*,$ $P_2^*,$ $P_3^*$ and
$P_4^* = \left(\rho, \gamma^{-1}, \rho - \beta^{-1}\right).$ 
The fixed points $P_1^*$ and $P_2^*$ are saddles with 
$\mathrm{dim}\,W^u(P_1^*,P_2^*) = 1,$ 
the point 
$P_3^*$ is an unstable spiral-sink node-source and
$P_4^*$ is an unstable spiral-source node-sink. 
Here, due to the unstable nature of all fixed points, 
fluctuating coexistence of all of the species is found. 
As we will see in Section~\ref{SRGD}, this coexistence can be 
governed by periodic or chaotic fluctuations.
\item 
$
    (\beta, \gamma, \mu) \in (3, 5] \times [5, 9.4] \times
             \left(\frac{\beta}{\beta-2}, 3\right).
$\\
In this zone the system has four fixed points: 
$P_1^*,$ $P_2^*,$ $P_3^*$ and $P_4^*.$
$P_1^*$ and $P_2^*$ are a saddles with
$\mathrm{dim}\,W^u(P_1^*,P_2^*) = 1$, 
$P_3^*$ is an unstable spiral-node source and
$P_4^*$ is an unstable spiral-source node-sink. 
The expected coexistence dynamics here are like those of zone F above.
\item 
$
   (\beta, \gamma, \mu) \in \mathsf{H}_4 \times
          \left(3,\psi_4(\beta,\gamma)\right).
$\\
In this region the system has four fixed points: 
$P_1^*,$ $P_2^*,$ $P_3^*$ and $P_4^*.$
$P_1^*$ and $P_2^*$ are saddles 
$\mathrm{dim}\,W^u(P_1^*,P_2^*) = 1,$ 
$P_3^*$ is an unstable spiral-sink node-source and
$P_4^*$ is a locally asymptotically stable sink of spiral-node type.
The dynamics here are the same as the ones in zone E.
\item 
$
      (\beta, \gamma, \mu) \in (2.5, 3) \times [5, 9.4] \times
             \left(\max\left\{3, \psi_4(\beta,\gamma)\right\}, 
                     \min\left\{4,\frac{\beta}{\beta-2}\right\}\right).
$\\
In this zone the system has four fixed points: 
$P_1^*,$ $P_2^*,$ $P_3^*$ and $P_4^*.$
$P_1^*$ is a saddle with 
$\mathrm{dim}\,W^u(P_1^*) = 1$, 
$P_2^*$ is a saddle with 
$\mathrm{dim}\,W^u(P_2^*) = 2$, 
$P_3^*$ is an unstable spiral-sink node-source and
$P_4^*$ is an unstable spiral-source node-sink. 
Here the dynamics can be also governed by coexistence among 
the three species via oscillations.
\item 
$
  (\beta, \gamma, \mu) \in \left(\tfrac{8}{3}, 5\right] \times [5, 9.4] \times 
     \left(\max\left\{3, \frac{\beta}{\beta-2}\right\}, 4\right).
$\\
In this zone the system has four fixed points: 
$P_1^*,$ $P_2^*,$ $P_3^*$ and $P_4^*.$
The fixed point $P_1^*$ is a saddle with 
$\mathrm{dim}\,W^u(P_1^*) = 1,$ 
the point $P_2^*$ is a saddle with 
$\mathrm{dim}\,W^u(P_2^*) = 2$, 
$P_3^*$ is an unstable spiral-node source and
$P_4^*$ is an unstable spiral-source node-sink. Dynamics here can
also be governed by all-species fluctuations, 
either periodic or chaotic.
\end{enumerate}
\end{theorem}

Figure~\ref{diag} provides a summary of the changes in the existence
and local stability of the fixed points for each one of the zones identified.
Also, we provide an animation of the dynamical outcomes tied to
crossing the cuboid following the direction of the 
dashed thick blue arrow represented in Figure~\ref{SuPeRfIgUrE}.
Specifically, the file \textsf{Movie-3.mp4} in the Supplementary
Material displays the dynamics along this line for variable $x_n$ 
(as a function of the three running parameters labelled factor),
as well as in the phase space $(x, y)$ and $(x, y, z).$

\section{Some remarks on global dynamics}\label{SRGD}

In this section we study the global dynamics in Zones~A and B from
the preceding section.

\begin{theorem}[Global dynamics in Zone A]\label{GAS}
Assume that $\mu < 1$ and let $(x,y,z)$ be a point from $\SD.$
Then,
\[ \lim_{n\to\infty} T^n(x,y,z) = (0,0,0) = P_1^*. \]
\end{theorem}

In what follows, $\lambda_{\mu}(\sigma) := \mu \sigma(1-\sigma)$
will denote the logistic map.

\begin{proof}[Proof of Theorem~\ref{GAS}]
From Figure~\ref{SuPeRfIgUrE} (or Lemmata~\ref{lem:P1}--\ref{lem:P4})
it follows that $(0,0,0)$ is the only fixed point of $T$ whenever
$\mu < 1$ and it is locally asymptotically stable.

We denote $(x_0,y_0,z_0) = (x,y,z) \in \SD$ and
$(x_n,y_n,z_n) = T^n(x,y,z) \in \RD$ for every $n \ge 1.$
Assume that there exists $n \ge 0$ such that $y_n = 0.$
Then, substituting $(x_n,0,z_n)$ into Equations~\eqref{eq:sistema} 
it follows that $y_{n+1} = z_{n+1} = 0$ and so
\[ T^{n+1}(x,y,z) = (x_{n+1}, 0, 0) \in [0,1] \times \{0\} \times \{0\}, \]
and
$T^{n+1+k}(x,y,z) = (\lambda_{\mu}^k(x_{n+1}), 0, 0)$
for every $k \ge 0.$
Since, $\mu < 1,$ one gets
$\lim_{k\to\infty} \lambda_{\mu}^k(\sigma) = 0$
for every $\sigma \in [0,1].$
So, the proposition holds in this case.

In the rest of the proof we assume that $y_n > 0$ for every $n \ge 0.$
We claim that
\[ x_n \le \frac{\mu^n}{4} \]
for every $n \ge 1.$
Let us prove the claim.
Since $(x,y,z) \in \SD \subset \RD$
(Proposition~\ref{prop:domain-walledsimplex}) with $y > 0$ we have
$x \ge z\ge 0$ and $x+y+z \le 1.$ 
Thus,
\[ x_1 = \mu x (1 - x - y - z) \le \mu x (1 - x) \le \frac{\mu}{4}, \]
which proves the case $n = 1.$
Assume now that the claim holds for some $n \ge 1$
and prove it for $n+1.$
As before, $(x_n,y_n,z_n) \in \RD$ with $y_n > 0$ implies
$x_n \ge z_n\ge 0$ and $x_n+y_n+z_n \le 1.$ 
Hence,
\[
    x_{n+1} = \mu x_n (1 - x_n - y_n - z_n) \le 
              \mu x_n \le 
              \mu \frac{\mu^n}{4} = 
              \frac{\mu^{n+1}}{4}.
\]
On the other hand, by using again the assumption that
$(x_n,y_n,z_n) \in \RD$ with $1 \ge y_n > 0$ for every $n \ge 0,$
and the definition of $T$ in~\eqref{eq:sistema}, 
we get that $x_n>z_n \ge 0$ for every $n\ge 0.$
Moreover,
$y_{n+1} = \beta y_n (x_n - z_n) \le \beta y_n x_n \le \beta x_n.$
Hence, for every $n \ge 0,$
\[
   0 \le z_n < x_n \le \frac{\mu^n}{4}
   \quad\text{and}\quad
   0 < y_n \le \beta x_{n-1} \le \beta \frac{\mu^{n-1}}{4}.
\]
This implies that
$\lim_{n\to\infty} (x_n,y_n,z_n) = (0,0,0)$
because $\mu < 1.$
\end{proof}

To study the global dynamics in Zone~B we need three simple lemmas.
The first one is on the logistic map; the second one relates the
first coordinate of the image of $T$ with the logistic map; the third
one is technical.

\begin{lemma}[On the logistic map]\label{lem:logisticmap}
Let $1 < \mu < 2$ and set
$I_0 := \bigl[\alpha_{\mu}, \widetilde{\alpha}_{\mu}\bigr]$ 
where
$0 < \alpha_{\mu} = 1 -\tfrac{1}{\mu} < \tfrac{1}{2}$
is the stable fixed point of $\lambda_{\mu}$ and
$\tfrac{1}{2} < \widetilde{\alpha}_{\mu} < 1$ 
is the unique point such that
$\lambda_{\mu}\bigl(\widetilde{\alpha}_{\mu}\bigr) = \alpha_{\mu}.$
Set also 
$I_{n+1} := \lambda_{\mu}\bigl(I_n\bigr) \subset I_n$ 
for every $n \ge 0.$
Then, for every $\varepsilon > 0$ there exists $N \ge 1$
such that
$
  I_{N} \subset \left[\alpha_{\mu}, \alpha_{\mu}+\varepsilon\right).
$
\end{lemma}

\begin{proof}
The fact that
$0 < \alpha_{\mu} = 1 -\tfrac{1}{\mu} < \tfrac{1}{2}$
is a stable fixed point of $\lambda_{\mu}$ for $1 < \mu < 2$
is well known. Also, since
$\lambda_{\mu}\Bigr\rvert_{\left[\alpha_{\mu},\tfrac{1}{2}\right]}$
is increasing and $1 < \mu < 2,$ it follows that
\[
   \alpha_{\mu} = \lambda_{\mu}\left(\alpha_{\mu}\right) < 
                  \dots <
                  \lambda_{\mu}^{n+1}\left(\tfrac{1}{2}\right) <
                  \lambda_{\mu}^n\left(\tfrac{1}{2}\right) < \dots <
                  \lambda_{\mu}^2\left(\tfrac{1}{2}\right) <
                  \lambda_{\mu}\left(\tfrac{1}{2}\right) < \tfrac{1}{2}.
\]
Therefore,
\[
   I_{1} = \lambda_{\mu}\bigl(I_0\bigr) =
           \lambda_{\mu}\left(\left[\alpha_{\mu}, \tfrac{1}{2}\right]\right) =
           \left[\alpha_{\mu}, \lambda_{\mu}\left(\tfrac{1}{2}\right)\right] \subset
           \left[\alpha_{\mu}, \tfrac{1}{2}\right] \subset I_0
\]
and, for every $n \ge 1,$ one gets
\[
   I_{n+1} = \lambda_{\mu}\bigl(I_n\bigr) =
             \left[\alpha_{\mu}, \lambda_{\mu}^{n+1}\left(\tfrac{1}{2}\right)\right] \varsubsetneq
             \left[\alpha_{\mu}, \lambda_{\mu}^{n}\left(\tfrac{1}{2}\right)\right] =
             I_{n} \subset 
             \left[\alpha_{\mu}, \tfrac{1}{2}\right].
\]
Then, the lemma follows from the fact that
$\lim_{n\to\infty} \lambda_{\mu}^n\left(\tfrac{1}{2}\right) = \alpha_{\mu}.$
\end{proof}

\begin{lemma}\label{lem:Tlogistikish}
Let $(x_0,y_0,z_0) \in \SD$ and set
\[ (x_n,y_n,z_n) := T(x_{n-1},y_{n-1},z_{n-1}) \in \RD \]
for every $n \ge 1.$
Then, for every $n \ge 1,$
\[
  0 \le x_n = \lambda_{\mu}(x_{n-1}) - \mu x_{n-1} (y_{n-1} + z_{n-1}) \le 
              \lambda_{\mu}(x_{n-1})
\]
and, when $\mu \le 2,$ it follows that
$0 \le x_{n} \le \lambda^n_{\mu}(x_0) \le \tfrac{1}{2}.$
\end{lemma}

\begin{proof}
The first statement is a simple computation:
\begin{align*}
x_n &= \mu x_{n-1} (1 - x_{n-1}) - \mu x_{n-1} (y_{n-1} + z_{n-1}) \\
    &= \lambda_{\mu}(x_{n-1}) - \mu x_{n-1} (y_{n-1} + z_{n-1})
     \le \lambda_{\mu}(x_{n-1})
\end{align*}
(notice that $\mu, x_n, x_{n-1}, y_{n-1}, z_{n-1} \ge 0$ because
$(x_{n},y_{n},z_{n}) \in \RD$ for every $n$).

The second statement for $n = 1$ follows directly from the first
statement and from the fact that
$
  \lambda_{\mu}([0,1]) =
         \lambda_{\mu}\left(\left[0,\tfrac{1}{2}\right]\right) \subset
         \left[0,\tfrac{1}{2}\right]
$
whenever $\mu \le 2.$

Assume now that the second statement holds for some $n \ge 1.$
Then, from the first statement of the lemma and the fact that
$\mu \le 2$ we have
\[
  0 \le x_{n+1} \le \lambda_{\mu}(x_{n}) \le
                    \lambda_{\mu}\left(\lambda^n_{\mu}(x_0)\right) =
                    \lambda^{n+1}_{\mu}(x_0) \le \lambda_{\mu}\left(\tfrac{1}{2}\right) \le 
                    \tfrac{1}{2},
\]
because
$\lambda_{\mu}\bigr\rvert_{\left[0,\tfrac{1}{2}\right]}$
is increasing.
\end{proof}

The proof of the next technical lemma is a simple exercise.

\begin{lemma}[The damped logistic map]\label{lem:DampedLogisticmap}
Let 
$\lambda_{\mu,s}(\sigma) := s\lambda_{\mu}(\sigma) = \mu s\sigma(1-\sigma)$ 
denote the \emph{damped logistic map} defined on the interval $[0,1].$
Assume that $1 < \mu < 2$ and $\tfrac{1}{\mu} < s < 1.$
Then the following properties of the damped logistic map hold:
\begin{enumerate}[(a)]
\item $\lambda_{\mu,s}(\sigma) < \lambda_{\mu}(\sigma)$ 
      for every $0 < \sigma < 1.$
\item $\lambda_{\mu,s}(0) = 0$ and
      $\lambda_{\mu,s}\Bigr\rvert_{\left[0,\tfrac{1}{2}\right]}$ 
      is strictly increasing.
\item $\lambda_{\mu,s}$ has exactly one stable fixed point
      $\alpha_{\mu,s} := 1-\tfrac{1}{\mu s}$
      with derivative
      \[ 
         \lambda'_{\mu,s}(\alpha_{\mu,s}) = 
               \lambda'_{\mu,s}(\sigma)\bigr\rvert_{\sigma=\alpha_{\mu,s}} = 
               \mu s(1-2\sigma)\bigr\rvert_{\sigma=\alpha_{\mu,s}} = 
               2 - \mu s < 1\, . 
      \]
\item For every $\sigma \in \bigl(0, \alpha_{\mu,s}\bigr)$ we have
\[ \sigma < \lambda_{\mu,s}(\sigma) < \lambda^2_{\mu,s}(\sigma) < \dots < \alpha_{\mu,s} \]
and $\lim_{k\to\infty} \lambda^k_{\mu,s}(\sigma) = \alpha_{\mu,s}.$
\end{enumerate}
\end{lemma}

\begin{theorem}[Global dynamics in Zone B]\label{GAS2}
Assume that $1 < \mu < \tfrac{\beta}{\beta-1}$ and let 
$(x,y,z)$ be a point from $\SD.$
Then, either $T^n(x,y,z) = (0,0,0)$ for some $n \ge 0$ or
\[ \lim_{n\to\infty} T^n(x,y,z) = \left(1-\mu^{-1},0,0\right) = P_2^*. \]
\end{theorem}

\begin{remark}
From Lemma~\ref{lem:existence:fixpoints} it follows that
the unique fixed points which exist in this case are
$P_1^*$ and $P_2^*.$
\end{remark}

\begin{proof}
From Figure~\ref{SuPeRfIgUrE} 
(or Lemmata~\ref{lem:P1}--\ref{lem:DampedLogisticmap})
it follows that 
$\left(\alpha_{\mu},0,0\right)$ with $\alpha_{\mu} := 1-\tfrac{1}{\mu}$ 
is the only locally asymptotically stable fixed point of $T.$
In the whole proof we will consider that
$\alpha_{\mu}$ is the unique stable fixed point of $\lambda_{\mu}.$
As in previous proofs, we denote $(x_0,y_0,z_0) = (x,y,z) \in \SD$ and
$(x_n,y_n,z_n) = T^n(x,y,z) \in \RD$ for every $n \ge 1.$

If there exists $n \ge 0$ such that $(x_n,y_n,z_n) = (0,0,0)$
we are done. Thus, in the rest of the proof we assume that
$(x_n,y_n,z_n) \ne (0,0,0)$ for every $n \ge 0.$

Assume that there exists $n \ge 0$ such that $y_n = 0.$
By the definition of $T$, it follows that
$
   T^{n+1}(x,y,z) = (x_{n+1}, 0, 0) \in [0,1] \times \{0\} \times \{0\},
$
and, consequently,
$
  T^{n+1+k}(x,y,z) = (\lambda_{\mu}^k(x_{n+1}), 0, 0)
$
for every $k \ge 0.$
Thus, since 
\[ 1 < \mu < \tfrac{\beta}{\beta-1} \le \tfrac{5}{3} < 2, \]
it turns out that
$\lim_{k\to\infty} \lambda_{\mu}^k(x_{n+1}) = \alpha_{\mu}$
(recall that we are in the case
$(x_n,y_n,z_n) \ne (0,0,0)$ for every $n \ge 0$
and, consequently,
$\lambda_{\mu}^k(x_{n+1}) \ne 0$ for every $k \ge 0$).
So, the proposition holds in this case.

In the rest of the proof we are left with the case 
$(x_n,y_n,z_n) \in \RD$ and $y_n > 0$ for every $n \ge 0.$
Moreover, suppose that $x_n = 0$ for some $n \ge 0.$
Since $(x_n,y_n,z_n) \in \RD$ we have that 
$0 \le z_n \le x_n = 0$ implies $z_n = 0.$
Consequently,
$(x_{n+1},y_{n+1},z_{n+1}) = T(x_n,y_n,z_n) = (0,0,0),$
a contradiction.
Thus, $x_n, y_n > 0$ for every $n \ge 0.$

Observe that, since $\mu < \tfrac{\beta}{\beta-1}$ we have
\[
  \alpha_{\mu} = \frac{\mu-1}{\mu} <
                 \frac{\mu-1}{\mu}\Biggr\rvert_{\mu=\tfrac{\beta}{\beta-1}} =
                 \tfrac{1}{\beta}.
\]
On the other hand,
$
  \lambda'_{\mu}\bigl(\alpha_{\mu}\bigr) = 
       \mu(1-2x)\Bigr\rvert_{x=\tfrac{\mu-1}{\mu}} = 2-\mu < 1
$
because $\mu > 1.$
Thus, there exist $r \in (2-\mu, 1)$ and $0 < \delta < \alpha_{\mu}$
such that
$\alpha_{\mu} + \delta < \tfrac{1}{\beta} \le \tfrac{2}{5} < \tfrac{1}{2},$ 
and $\lambda'_{\mu}(x) < r$ for every 
$x \in \bigl(\alpha_{\mu}-\delta, \alpha_{\mu}+\delta\bigr).$

Set $\tau := \beta \bigl(\alpha_{\mu}+\delta\bigr) < 1.$
To show that
$\lim_{n\to\infty} \bigl(x_n, y_n, z_n\bigr) = \left(\alpha_{\mu},0,0\right)$
we will prove that the following two statements hold:
\begin{enumerate}[(i)]
\item There exists a positive integer $N$ such that
\[ 0 \le y_n < \tau^{n-N} \quad\text{and}\quad 0 \le z_n < \gamma\tau^{n-1-N}, \]
for every $n \ge N+2.$
\item For every $0 < \varepsilon < \delta$ there exists a positive
integer $M$ such that $\abs{x_n - \alpha_{\mu}} < \varepsilon$
for all $n \ge M.$
\end{enumerate}

To prove (i) and (ii) we fix $0 < \varepsilon < \delta < \alpha_{\mu}$
and we claim that there exists a positive integer $N = N(\varepsilon)$
such that $x_n < \alpha_{\mu} + \varepsilon$ for every $n \ge N.$
Now we prove the claim.
Assume first that
$x_0 \in \bigl[0,\alpha_{\mu}] \cup \bigl[\widetilde{\alpha}_{\mu},1\bigr],$
where 
$\tfrac{1}{2} < \widetilde{\alpha}_{\mu} < 1$
is the unique point such that
$\lambda_{\mu}\bigl(\widetilde{\alpha}_{\mu}\bigr) = \alpha_{\mu}.$
Since
\[
  \lambda_{\mu}\left(\bigl[0,\alpha_{\mu}] \cup 
        \bigl[\widetilde{\alpha}_{\mu},1\bigr]\right) =
            \lambda_{\mu}\left(\bigl[0,\alpha_{\mu}]\right) =
            \bigl[0,\alpha_{\mu}],
\]
$\lambda^n_{\mu}(x_0) \in \bigl[0, \alpha_{\mu}\bigr]$
for every $n \ge 1.$ 
Thus, if we set $N = N(\varepsilon) = 1$ and we take $n \ge N$, 
by Lemma~\ref{lem:Tlogistikish} we have
\[
   0 \le x_n \le \lambda^n_{\mu}(x_0) \le 
                 \alpha_{\mu} < 
                 \alpha_{\mu} + \varepsilon.
\]
Assume now that
$x_0 \in \bigl(\alpha_{\mu}, \widetilde{\alpha}_{\mu}\bigr).$
By Lemmas~\ref{lem:Tlogistikish} and~\ref{lem:logisticmap},
there exists $N = N(\varepsilon) \ge 1$ such that
\[
   0 \le x_n \le \lambda^n_{\mu}(x_0) \in 
                 I_n \subset
                 I_{N} \subset 
                 \left[\alpha_{\mu}, \alpha_{\mu}+\varepsilon\right)
\]
for every $n \ge N.$
This ends the proof of the claim.

Now we prove (i).
From the above claim we have
\begin{equation}\label{eq:xbound}
   \beta x_n < \beta \bigl(\alpha_{\mu}+\varepsilon\bigr) < 
               \beta \bigl(\alpha_{\mu}+\delta\bigr) = \tau < 1
   \quad\text{for every $n \ge N.$}
\end{equation}
Consequently, 
by the iterative use of \eqref{eq:xbound},
for every $n \ge N+2$ we have
\begin{multline*}
  y_n = \beta y_{n-1}(x_{n-1}-z_{n-1}) \le 
        \beta y_{n-1} x_{n-1} <
        \tau y_{n-1} < \\
        \tau^2 y_{n-2} < \cdots <
        \tau^{n-N} y_{_N} \le 
        \tau^{n-N},
\end{multline*}
and
$
  z_n = \gamma y_{n-1} z_{n-1} \le \gamma y_{n-1} < \gamma \tau^{n-1-N}.
$

Now we prove (ii). In this proof we will use the damped logistic map
$\lambda_{\mu, s}$ with parameter
$1 > s > \tfrac{1}{\mu\varepsilon+1}.$
From (i) it follows that there exists a positive integer
$\widetilde{M} \ge N+2$
such that
\[
   y_n + z_n < \min\left\{ 
        \frac{\beta(1-r)}{\mu\tau} \varepsilon, 
        (1-s)\bigl(1- (\alpha_{\mu}-\varepsilon)\bigr) 
   \right\}
\]
for every $n \ge \widetilde{M}.$
Observe that if there exists $M \ge \widetilde{M}$ such that
$\abs{x_M - \alpha_{\mu}} < \varepsilon,$
then
$\abs{x_n - \alpha_{\mu}} < \varepsilon$ for every $n \ge M.$
To prove it assume that there exists $n \ge M$ such that
$\abs{x_k - \alpha_{\mu}} < \varepsilon$ for $k = M, M+1,\dots, n$
and prove it for $n+1.$
By Lemma~\ref{lem:Tlogistikish}, Equation~\eqref{eq:xbound} and the 
Mean Value Theorem,
\begin{align*}
  \abs{x_{n+1} - \alpha_{\mu}}
     = &\ \abs{\lambda_{\mu}\bigl(x_n\bigr) - \alpha_{\mu} - 
                                      \mu x_n \bigl(y_n + z_n\bigr)} \le\\
       &\ \abs{\lambda_{\mu}\bigl(x_n\bigr) - 
                \lambda_{\mu}\bigl(\alpha_{\mu}\bigr)} + 
                \mu x_n \bigl(y_n + z_n\bigr) = \\
       &\ \lambda'_{\mu}(\xi) \abs{x_n - \alpha_{\mu}} + 
                \mu x_n \bigl(y_n + z_n\bigr) <\\
       &\ \lambda'_{\mu}(\xi) \varepsilon + 
                \mu\frac{\tau}{\beta}\frac{\beta(1-r)}{\mu\tau} \varepsilon =
                \varepsilon\left(\lambda'_{\mu}(\xi) + (1-r)\right),
\end{align*}
where $\xi$ is a point between $x_n$ and $\alpha_{\mu}.$
Since
$
 \abs{\xi - \alpha_{\mu}} \le \abs{x_n - \alpha_{\mu}} < \varepsilon < \delta
$
it follows that $\lambda'_{\mu}(\xi) < r.$ So,
$
  \abs{x_{n+1} - \alpha_{\mu}} < 
      \varepsilon\left(\lambda'_{\mu}(\xi) + (1-r)\right) < 
      \varepsilon.
$

To end the proof of the proposition we have to show that
there exists $M \ge \widetilde{M}$ such that
$\abs{x_M - \alpha_{\mu}} < \varepsilon.$
By the above claim we know that
$x_{\widetilde{M}} < \alpha_{\mu} + \varepsilon.$
So, the statement holds trivially with $M = \widetilde{M}$ whenever
$x_{\widetilde{M}} > \alpha_{\mu} - \varepsilon.$

In the rest of the proof we may assume that
$0 < x_{\widetilde{M}} \le \alpha_{\mu} - \varepsilon.$
Observe that
$\varepsilon < \delta < \alpha_{\mu} = \tfrac{\mu-1}{\mu}$
implies $\mu\varepsilon < \mu-1,$ which is equivalent to
$\mu\varepsilon+1 < \mu$ and, consequently,
$
  \tfrac{1}{\mu} < \tfrac{1}{\mu\varepsilon+1} < s < 1.
$
So, $s$ verifies the assumptions of Lemma~\ref{lem:DampedLogisticmap}.
Moreover, since $\tfrac{1}{\mu\varepsilon+1} < s,$
we have
\[
  1 < s(\mu\varepsilon+1) \Longleftrightarrow 
      \mu s-1 > \mu s-\mu s \varepsilon - s \Longleftrightarrow 
      \mu(\mu s - 1) > 
      \mu s \bigl(\mu(1-\varepsilon)-1\bigr),
\]
which is equivalent to
$
  \alpha_{\mu,s} = \tfrac{\mu s-1}{\mu s} > 
                   \tfrac{\mu(1-\varepsilon)-1}{\mu} 
                 = \alpha_{\mu} - \varepsilon.
$
Summarizing, we have,
$
  0 < x_{\widetilde{M}} \le \alpha_{\mu} - \varepsilon < \alpha_{\mu,s}.
$

By Lemma~\ref{lem:DampedLogisticmap}(d),
there exists $L > 0$ such that
$
  \lambda^L_{\mu,s}\bigl(x_{\widetilde{M}}\bigr) > \alpha_{\mu} - \varepsilon.
$
If there exists 
$N < \widetilde{M} < M < \widetilde{M}+L$ such that
$x_{M} > \alpha_{\mu} - \varepsilon$ then,
$\abs{x_M - \alpha_{\mu}} < \varepsilon$ because, 
by the above claim, 
$x_{M} < \alpha_{\mu} + \varepsilon.$
Hence, we may assume that
$x_{\widetilde{M}+k} \le \alpha_{\mu} - \varepsilon$
for every $k = 0,1,\dots, L-1.$
Then,
\[
   \mu x_{\widetilde{M}} \bigl(y_{\widetilde{M}} + z_{\widetilde{M}}\bigr) <
      \mu x_{\widetilde{M}} (1-s) \bigl(1- (\alpha_{\mu}-\varepsilon)\bigr) \le
      (1-s) \mu x_{\widetilde{M}} \bigl(1-x_{\widetilde{M}}\bigr) =
     (1-s) \lambda_{\mu}\bigl(x_{\widetilde{M}}\bigr),
\]
which,
by Lemmas~\ref{lem:Tlogistikish} and~\ref{lem:DampedLogisticmap}(b),
is equivalent to
\[
   0 < \lambda_{\mu,s}\bigl(x_{\widetilde{M}}\bigr) =
       s \lambda_{\mu}\bigl(x_{\widetilde{M}}\bigr) <
       \lambda_{\mu}\bigl(x_{\widetilde{M}}\bigr) - 
             \mu x_{\widetilde{M}} \bigl(y_{\widetilde{M}} + 
             z_{\widetilde{M}}\bigr) =
       x_{\widetilde{M}+1}\,.
\]
Moreover, by iterating these computations and
using again Lemma~\ref{lem:DampedLogisticmap}(b) we have
\[
   0 < \lambda^2_{\mu,s}\bigl(x_{\widetilde{M}}\bigr) <
   \lambda_{\mu,s}\bigl(x_{\widetilde{M}+1}\bigr) < 
   x_{\widetilde{M}+2}
\]
(notice that $x_{\widetilde{M}+1} < \tfrac{1}{2}$ by 
Lemma~\ref{lem:Tlogistikish}).
Thus, by iterating again all these computations we get
$
  0 < \lambda^k_{\mu,s}\bigl(x_{\widetilde{M}}\bigr) < x_{\widetilde{M}+k}
$
for every $k = 0,1,\dots, L.$ 
This implies that
\[
  \alpha_{\mu} - \varepsilon <
       \lambda^L_{\mu,s}\bigl(x_{\widetilde{M}}\bigr) < 
       x_{\widetilde{M}+L},
\]
and the statement holds with $M = \widetilde{M}+L.$
\end{proof}

\begin{figure}[htp]
\includegraphics[width=\textwidth]{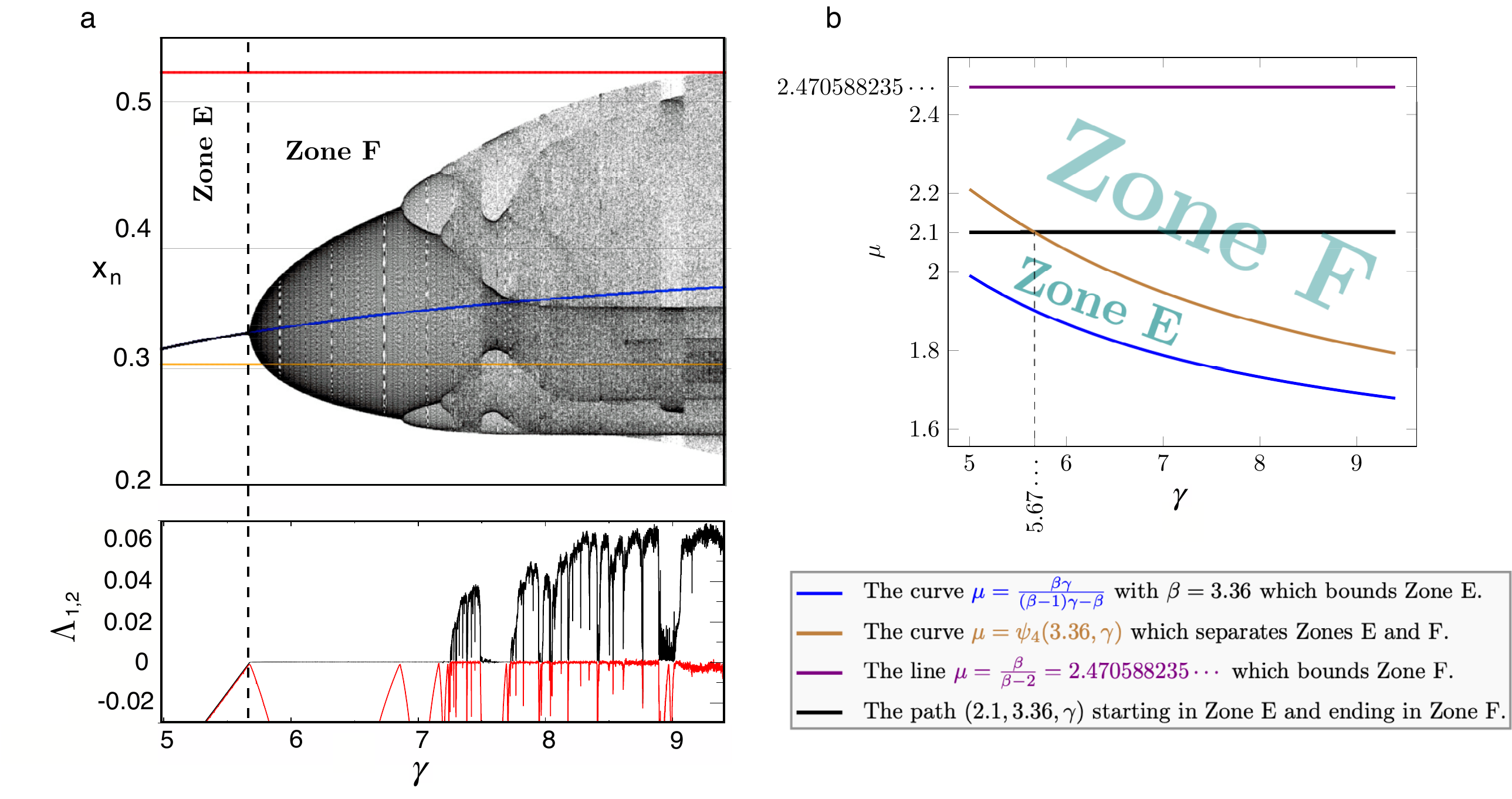}
\captionsetup{width=\linewidth}
\caption{
(a, upper) Bifurcation diagram displaying the dynamics of preys 
$x$ at increasing the predation intensity of predator $z$ on 
predator $y$, given by $\gamma$, using 
$\mu = 2.1$ and $\beta = 3.36.$ 
This range of $\gamma$ covers zones E and F, separated by the
vertical dashed line. 
The values of the fixed points are shown overlapped, with 
$P^*_2$: red; 
$P^*_3$: orange; and 
$P^*_4$: blue.
(a, lower) Spectrum of Lyapunov exponents, $\Lambda_{1,2,3}$ computed
for the same range of $\gamma$ used in the bifurcation diagram 
(for clarity only $\Lambda_{1,2}$ are displayed, in black and red respectively).
In both panels the initial conditions are: 
$x_0 = 0.1$, $y_0 = 0.02$, and $z_0 = 0.03.$ 
(b) A cut of the parameter space at $\beta = 3.36$ showing the path 
$(\mu = 2.1, \beta = 3.36, \gamma)$ followed by the bifurcation diagram 
of (a). 
The dynamics for this parameter range can be visualised in the file 
\textsf{Movie-4.mp4} in the Supplementary Material, where the 
three-dimensional bifurcation diagram displayed in the next figure is shown, 
together with the attractors projected in the two-dimensional phase spaces 
$(x, y)$, $(x, z)$, $(y, z)$, as well with the full attractor in the phase 
space $(x, y, z).$}\label{bif_gamma}
\end{figure}

\section{Chaos and Lyapunov exponents}\label{LEs}
As expected, iteration of the map~\eqref{eq:sistema} suggests the
presence of strange chaotic attractors
(see Figures~\ref{bif_gamma_3d}(c,d) and~\ref{bif_beta}(e,f)).
In order to identify chaos we compute Lyapunov exponents,
labelled $\Lambda_i$,
using the computational method described in
\cite[pages 74--80]{Parker1989},
which provides the full spectrum of Lyapunov exponents
for the map~\eqref{eq:sistema}.
\begin{figure}
\includegraphics[width=\textwidth]{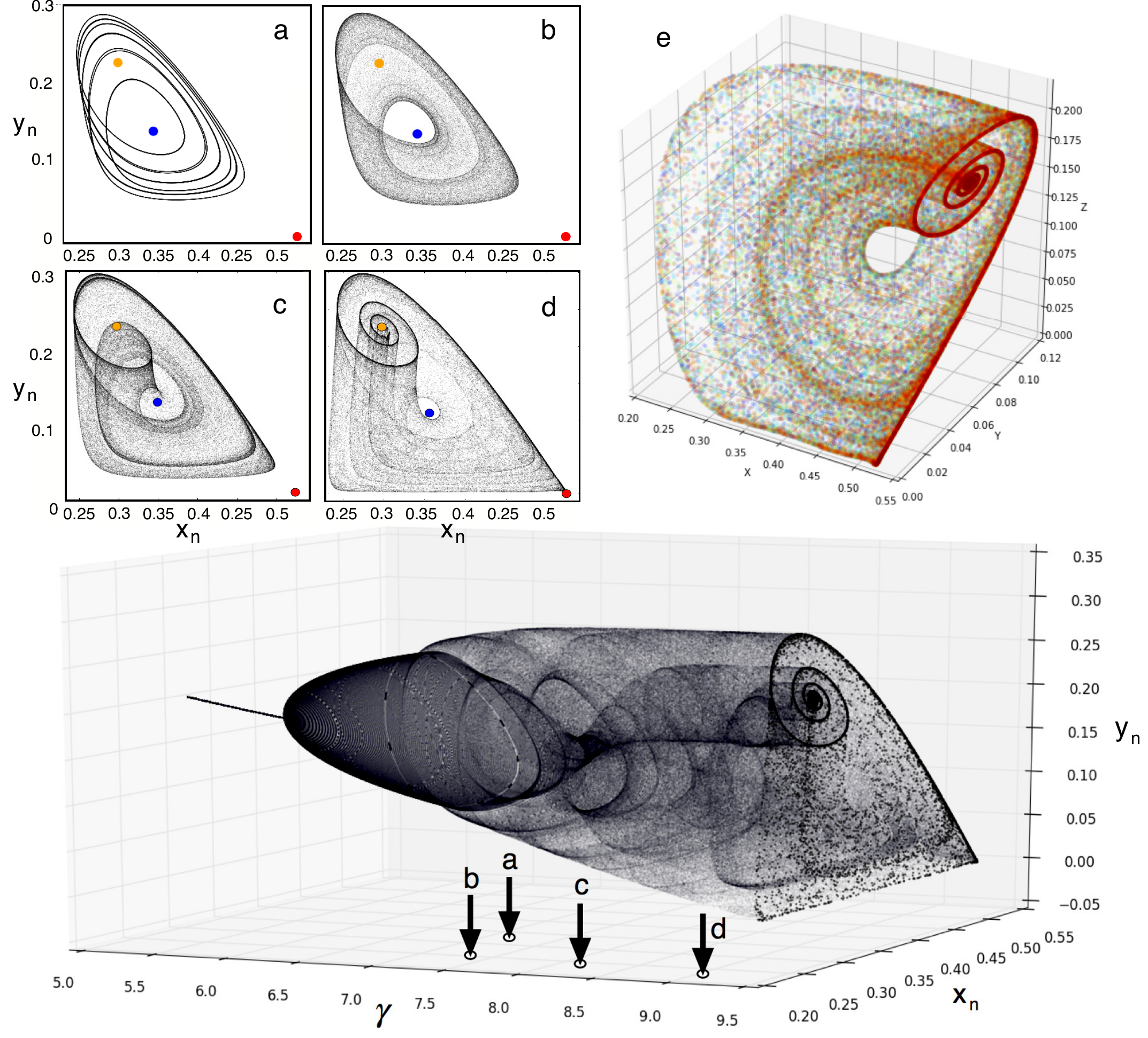}
\captionsetup{width=\linewidth}
\caption{Three-dimensional bifurcation diagram plotting the
population values $(x, y)$ using the predator rate of predator $z$ as
control parameter, setting $\mu = 2.1$ and $\beta = 3.36.$ 
The attractors above the bifurcation diagram are displayed using: 
(a) $\gamma = 7.3$; 
(b) $\gamma = 7.46$; 
(c) $\gamma = 8.14$, and 
(d) $\gamma = 9.14.$ 
All of the attractors are in zone $F.$ 
The fixed points are shown in the phase space, with 
$P_2^*$: red; 
$P_3^*$: orange; and 
$P_4^*$: blue. 
The initial conditions are the same than in Figure~\ref{bif_gamma}. 
In (e) we display the full chaotic attractor using $\gamma = 9.14.$ 
Here the color gradient corresponds to time: red dots are longer times. 
See also the file \textsf{Movie-4.mp4} in the Supplementary Material.}\label{bif_gamma_3d}
\end{figure}
\begin{figure}[htp]
\includegraphics[width=\textwidth]{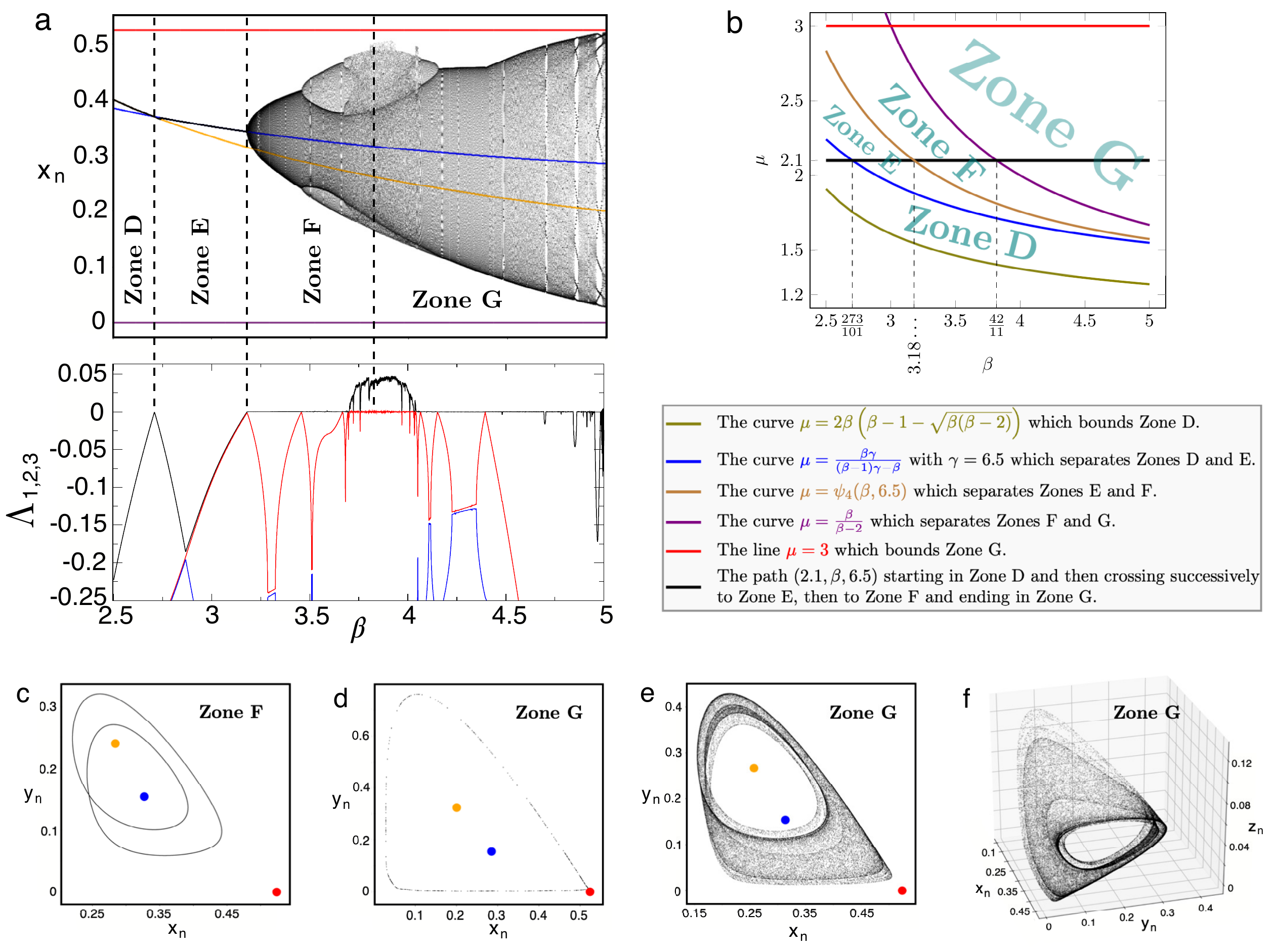}
\captionsetup{width=\linewidth}
\caption{
(a) Bifurcation diagram displaying preys' dynamics at increasing 
the rate of predation of predator $y$ (constant $\beta$) on the 
prey $x.$ 
The explored range of $\beta$ goes from zones D to G 
(changes between zones are indicated with vertical dashed lines).
Here the values of the fixed points when increasing $\beta$ are also
displayed ($P_2^*$: red; $P_3^*$: orange; and $P_4^*$: blue). 
Below we plot the spectrum of Lyapunov exponents, $\Lambda_i$, 
for the same range of $\beta.$ 
Here we fix $\mu = 2.1$ and $\gamma = 6.5.$ 
The initial conditions are the same than in the previous figure. 
(b) A cut of the parameter space at $\gamma = 6.5$ showing the path 
$(2.1, \beta, 6.5).$
Three attractors are shown with: 
(c) $\beta = 3.52$ (zone F);
(d) $\beta = 4.99$ and 
(e, f) $\beta = 3.89$ (zone G). 
The dynamics tied to the increase of $\beta$ can be visualised 
in the file \textsf{Movie-5.mp4} in the Supplementary Material.}\label{bif_beta}
\end{figure}

Let us explore the dynamics of the system focusing on the strength of
predation, parametrised by constants $\gamma$ and $\beta.$ To do so
we first investigate the dynamics at increasing the predation rate of
predator $z$ on predator $y$, given by $\gamma.$ We have built a
bifurcation diagram displaying the dynamics of the prey species $x$
by iterating Equations~\eqref{eq:sistema} at increasing $\gamma$, setting
$\mu = 2.1$ and $\beta = 3.36$ (see Figure~\ref{bif_gamma}(a)). The
increase in $\gamma$ for these fixed values of $\mu$ and $\beta$
makes the dynamics to change between zones $E \to F$
(see also Figure~\ref{bif_gamma}(b)).
For $5 < \gamma < 5.673555\cdots$, populations
achieve a static coexistence equilibrium at $P_4^*$, which is
achieved via damped oscillations (see the properties in zone E).
Increasing $\gamma$ involves the entry into zone F, where all of the
fixed points have an unstable nature and thus periodic and chaotic
solutions are found. Here we find numerical evidences of a route to
chaos driven by period-doubling of invariant closed curves that
appears after a supercritical Neimark-Sacker (Hopf-Andronov)
bifurcation for maps (flows)~\cite{Schuster1984,Kuznetsov1998},
for which the maximal Lyapunov exponent is zero
(see the range $5.673555 < \gamma \lesssim 7.25$),
together with complex eigenvalues for the fixed point $P_4^*$
(which is locally unstable).
Notice that the first Neimark-Sacker
bifurcation marks the change from zones E to F
(indicated with a vertical dashed line in Figure~\ref{bif_gamma}).
This means that an increase in the predation rate of species $y$
unstabilises the dynamics and the three species fluctuate chaotically.
Figure~\ref{bif_gamma_3d} displays the same bifurcation diagram than
in Figure~\ref{bif_gamma}, represented in a three-dimensional space
where it can be shown how the attractors change at increasing $\gamma$
projected onto the phase space $(x, y).$
Here we also display several projections of periodic
(Figure~\ref{bif_gamma_3d}a) and strange chaotic
(Figures~\ref{bif_gamma_3d}(b-d)) attractors.
Figure~\ref{bif_gamma_3d}e displays the full chaotic attractor.
For an animated visualisation of the dynamics dependence on $\gamma$ we
refer the reader to the file \textsf{Movie-4.mp4} in the
Supplementary Material.
\begin{figure}[htp]
\includegraphics[width=\textwidth]{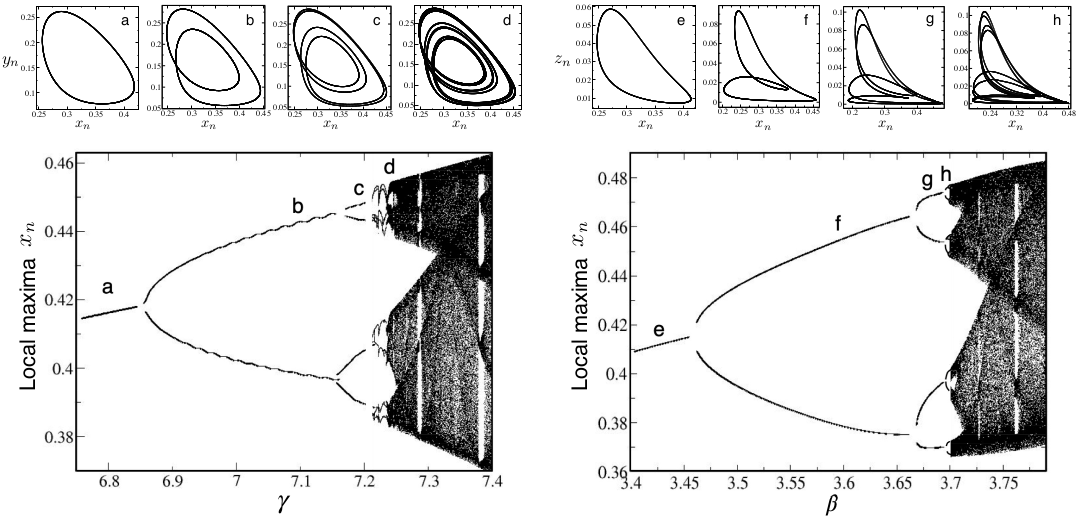}
\captionsetup{width=\linewidth}
\caption{Route to chaos at increasing predation rates governed by
period-doubling of invariant curves. We display the local maxima of
time series $x_n$ on the attractor for 
$\gamma$ (left diagram with $\beta = 3.36$) and 
$\beta$ (diagram at the right with $\gamma = 6.5$). 
Above the diagrams we display the attractors projected on the
phase space $(x, y)$ and $(x, z)$, with: 
(a) $\gamma = 6.8,$
(b) $\gamma = 7.1,$ 
(c) $\gamma = 7.18,$
(d) $\gamma = 7.21,$ 
(e) $\beta = 3.425$, 
(f) $\beta = 3.6,$, 
(g) $\beta = 3.685,$ and 
(h) $\beta = 3.7.$ 
In all plots the initial conditions are 
$x_0 = 0.2, y_0 = 0.02, z_0 = 0.03.$ 
See the file \textsf{Movie-6.mp4} (Supplementary Material) for a visualisation 
of the full attractor and the time series $x_n$, $y_n$, and $z_n$ undergoing 
period-doubling of closed curves tied to the bifurcations diagram at the left, 
shown within the range $6.75 \leq \gamma \leq 8.$}\label{routes}
\end{figure}

To further investigate the dynamics considering another key
ecological parameter, we study the dynamics at increasing the
predation strength of predator $y$ on preys $x$, which is given by
parameter $\beta.$ As an example we have selected the range 
$2.5 \leq \beta \leq 5$, which corresponds to one of the sides of 
$\Q.$ 
Here the range of $\beta$ follows the next order of crossing of the 
zones in $\Q$ when increasing $\beta$: $D \to E \to F \to G.$ 
Figure~\ref{bif_beta}(a) shows the bifurcation diagram also obtained 
by iteration. 
In Figure~\ref{fourier}(b) we also provide a diagram of the 
stability zones crossed in the bifurcation diagram.
Here, for $2.5 \leq \beta < 273/101$ the dynamics falls into zone D,
for which the top predator $z$ goes to extinction and the prey and
predator $y$ achieve a static equilibrium.
Increasing $\beta$ involves the entry into zone E
(at $\beta = 273/101$), the region where the fixed point of all-species
coexistence is asymptotically locally stable. 
Counter-intuitively, stronger predation of $y$ on $x$ makes the three 
species to coexist, avoiding the extinction of the top predator $z.$ 
At $\beta \approx 3.1804935$ there is another change to zone F, 
where all of the fixed points are unstable and thus periodic dynamics can occur. 
As we previously discussed, this is due to a series of bifurcations giving
place to chaos. We notice that further increase of $\beta$ involves
another change of zone. 
Specifically, at $\beta = 42/11 = 3.\overline{81}$ the dynamics changes 
from zone F to G.
Several attractors are displayed in Figure~\ref{bif_beta}:
(c) period-two invariant curve with $\beta = 3.52$ and zero maximal
Lyapunov exponent projected onto the phase space $(x, y)$, found in zone $F$; 
and two attractors of zone $G$, given by (d) a strange chaotic attractor with 
$\beta = 4.99$ and maximal Lyapunov exponent equals $0.0044\cdots$ also 
projected on $(x, y)$;
a strange chaotic attractor projected onto the phase space $(x, y)$ (e),
and in the full space space (f) with $\beta = 3.89$ and maximal Lyapunov
exponent equals $0.047\cdots.$
The file \textsf{Movie-5.mp4} displays the dynamics tied to the bifurcation
diagram displayed in Figure~\ref{bif_beta}.
\begin{figure}[htp]
\includegraphics[width=0.95\textwidth]{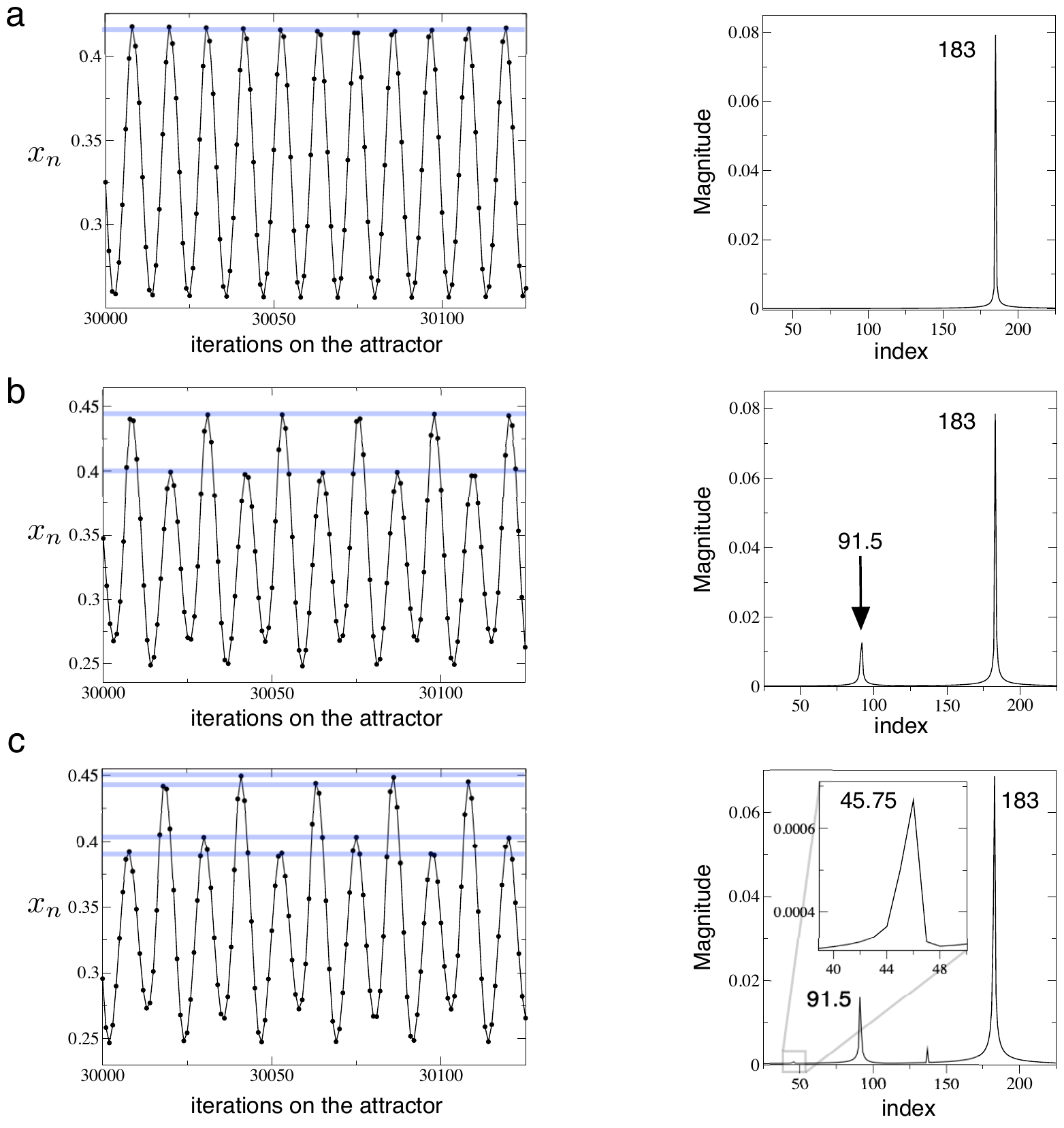}
\captionsetup{width=\linewidth}
\caption{Period-doublings of invariant curves represented with the
time series of the prey $x$ for the values of $\gamma$:
(a) $\gamma = 6.8,$
(b) $\gamma = 7.1,$ and
(c) $\gamma = 7.18$
(the same values of the left picture in Figure~\ref{routes}).
For better visualisation we have overlapped blue horizontal lines
indicating the maxima of the time series.
Note that in (c) the highest periods appear to be very close
(see also the attractor (c) in the previous figure).
The fast Fourier transforms for these three curves seem to show
numerical evidence of a period doubling phenomenon.
The FFT analysis of time series (c) contains an enlarged view of
the peak at index $45.75$, which is half the one found at $91.5$
and a quarter of $183$, all of them providing the most relevant
coefficients (their modulus, in fact) of their DFT and,
therefore, the main frequency of each discrete curve.}\label{fourier}
\end{figure}

\subsection{Route to chaos: Period-doubling of invariant curves}
It is known that some dynamical systems can enter into a chaotic
regime by means of different and well-defined routes
\cite{Schuster1984}. The most familiar ones are: 
  (i) the period-doubling route (also named Feigenbaum scenario); 
 (ii) the Ruelle-Takens-Newhouse route; 
(iii) and the intermittency route (also named Manneville-Pomeau route). 
The Feigenbaum scenario is the one identified in the logistic equation 
for maps, which involves a cascade of period doublings that ultimately 
ends up in chaos \cite{May1976}. 
The Ruelle-Takens-Newhouse involves the appearance of invariant curves 
that change to tori and then by means of tori bifurcations become unstable 
and strange chaotic attractors appear.
Finally, the intermittency route, tied to fold bifurcations, involves
a progressive appearance of chaotic transients which increase in
length as the control parameter is changed, finally resulting in a
strange chaotic attractor.

The bifurcation diagrams computed in
Figures~\ref{bif_gamma} and~\ref{bif_beta} 
seem to indicate that after a Neimark-Sacker bifurcation, 
the new invariant curves undergo period-doublings 
(see e.g., the beginning from Zone F until the presence of chaos 
in Figure~\ref{bif_beta}(a)). 
In order to characterise the routes to chaos at increasing the predation
parameters $\gamma$ and $\beta$, we have built bifurcation diagrams
by plotting the local maxima of time series for $x_n$ for each value
of these two parameters. 
The time series have been chosen after discarding a transient of 
$3\cdot10^4$ iterations to ensure that the dynamics lies in the attractor. 
The plot of the local maxima allows to identify
the number of maxima of the invariant curves as well as from the
strange attractors, resulting in one maximum for a period-1 invariant
curve, in two maxima for period-2 curves, etc. 
At the chaotic region the number of maxima appears to be extremely 
large (actually infinite). 
The resulting bifurcation diagram thus resembles the
celebrated period-doubling scenario of periodic points 
(Feigenbaum scenario).

The results are displayed in Figure~\ref{routes}. 
Since the system is discrete, the local maxima along the bifurcation 
diagrams have been smoothed using running averages. 
For both $\gamma$ and $\beta$, it seems clear that the invariant curves 
undergo period doubling. 
We also have plotted the resulting attractors for period-1,2,4,8 orbits
(see e.g. Figure~\ref{routes}(a-d) for the case with $\gamma$ using
projections in the $(x, y)$ phase space).

We have finally performed a Fast Fourier Transform (FFT) of the time
series for $x_n$ on the attractor corresponding to the attractors
displayed in~\ref{routes}(a-d)
and~\ref{routes}(e-h).
The FFT emphasizes the main frequencies (or periods) composing the
signal by showing the modulus of their Fourier coefficients. Remember
that FFT provides an efficient and fast way to compute the Discrete
Fourier Transform, DFT in short, of a discrete signal: 
given $x_0,x_1,\ldots,x_{N-1}$ complex numbers, 
its DFT is defined as the sequence $f_0,f_1,\ldots,f_{N-1}$ 
determined by
\[
  f_j = \sum_{k = 0}^{N-1} x_k
        \exp\left(\frac{-2\pi\text{i}jk}{N}\right).
\]

The FFTs have been computed using times series of $2^{11}$ points
after discarding the first $3 \cdot 10^4$ iterations of the map 
(a transitory). 
The results are displayed in Figure~\ref{fourier} for 
cases~\ref{routes}(a-d).
Similar results have been obtained for  cases~\ref{routes}(e-h, results not shown).
These FFT have been performed using a rectangular data window, 
and we have plotted the index of the signal versus its magnitude. 
It can be observed, by direct inspection, that the first relevant 
coefficient (in fact, its modulus) appear at each graph at half 
the index of the previous one (upper). 
This can be a numerical evidence of a period doubling 
(see also the animation in \textsf{Movie-6.mp4} in the Supplementary
Material to visualise the changes in the time series and in the
attractor at increasing $\gamma$). 
Here the period doubling of the curves can be clearly seen. 
A deeper study on the characterisation of this period-doubling 
scenario will be carried out in future work by computing the 
linking and rotation numbers of the curves.

\section{Conclusions}
The investigation of discrete-time ecological dynamical systems has
been of wide interest as a way to understand the complexity of
ecosystems, which are inherently highly nonlinear. Such
nonlinearities arise from density-dependent processes typically given
by intra- or inter- specific competition between species, by
cooperative interactions, or by antagonistic processes such as
prey-predator or host-parasite dynamics. Discrete models have been
widely used to model the population dynamics of species with
non-overlapping generations \cite{May1974,May1976,Schaffer1986}.
Indeed, several experimental research on insect dynamics revealed a
good matching between the observed dynamics and the ones predicted by
discrete
maps~\cite{Constantino1997,Desharnais2001,Dennis1997,Dennis2001}.

Typically, discrete models can display irregular or chaotic dynamics
even when one or two species are
considered~\cite{May1974,May1976,Elalim2012,Azmy2008,zhang2009}.
Additionally, the study of the local and global dynamics for
multi-species discrete models is typically performed numerically
(iterating) and most of the times fixing the rest of the parameters
to certain values. Hence, a full analysis within a given region of
the parameter space is often difficult due to the dimension of the
dynamical system and to the amount of parameters appearing in the
model. In this article we extend a previous two-dimensional map
describing predator-prey dynamics~\cite{Holden1986}. The extension
consists in including a top predator to a predator-prey model,
resulting in a three species food chain. This new model considers
that the top predator consumes the predators that in turn consume
preys. Also, the top predator interacts negatively with the growth of
the prey due to competition. Finally, the prey also undergoes
intra-specific competition.

We here provide a detailed analysis of local and global dynamics
of the model within a given volume of the full parameter space
containing relevant dynamics.
The so-called escaping set, causing sudden populations extinctions,
is identified.
These escaping sets contain zones which involve the surpass of the
carrying capacity and the subsequent extinction of the species.
For some parameter values these regions appear to have a complex,
fractal structure.

Several parametric zones are identified, for which
different dynamical outcomes exist: all-species extinctions,
extinction of the top predator, and persistence of the three species
in different coexistence attractors. Periodic and chaotic regimes are
identified by means of numerical bifurcation diagrams and of Lyapunov
exponents. We have identified a period-doubling route of invariant
curves chaos to chaos tuning the predation rates of both predators.
This route involves a supercritical Neimark-Sacker bifurcation giving
rise to a closed invariant curve responsible of all-species
coexistence. Despite this route to chaos has been found for given
combination of parameters and initial conditions tuning predation
rates, future work should address how robust is this route to chaos
to other parameter combinations. Interestingly, we find that this
route to chaos for the case of increasing predation directly on preys
(tuning $\beta$) can involve an unstable persistence of the whole
species via periodic or chaotic dynamics, avoiding the extinction of
top predators. This result is another example that unstable dynamics
(such as chaos) can facilitate species coexistence or survival, as
showed by other authors within the frameworks of homeochaotic
\cite{Kaneko1992,Ikegami1992} and metapopulation \cite{Allen1993}
dynamics.
\section*{Acknowledgements}
The research leading to these results has received funding from
``la Caixa'' Foundation and from a MINECO grant awarded to the
Barcelona Graduate School of Mathematics (BGSMath) under the ``Mar\'ia de Maeztu''
Program (grant MDM-2014-0445).
LlA has been supported by the Spain's "Agencial Estatal de Investigaci\'on" (AEI) grant MTM2017-86795-C3-1-P.
JTL has been partially supported by the MINECO/FEDER grant MTM2015-65715-P,
by the Catalan grant 2017SGR104,
and by the Russian Scientific Foundation grants 14-41-00044 and 14-12-00811.
JS has been also funded by a ``Ram\'on y Cajal'' Fellowship (RYC-2017-22243)
and by a MINECO grant MTM-2015-71509-C2-1-R and the AEI grant RTI2018-098322-B-100.
RS and BV have been partially funded by the Botin Foundation,
by Banco Santander through its Santander Universities Global Division
and by the PR01018-EC-H2020-FET-Open MADONNA project.
RS also acknowledges support from the Santa Fe Institute.
JTL thanks the Centre de Recerca Matem\`atica (CRM) for its hospitality
during the preparation of this work.

\end{document}